\documentclass[final,leqno]{siamltex}

\usepackage{amsmath,amssymb}
\usepackage{graphicx}
\usepackage{footnote}

\title{The Alignment Properties of Monge-Amp\`ere based Mesh Redistribution Methods: I Linear Features}

\author{C.J. Budd\thanks{University of Bath, UK, BA2 7AY (mascjb@bath.ac.uk).},\and
 R. D. Russell\thanks{Simon Fraser University, Burnaby, BC, Canada, V5N IS6 (rdr@sfu.ca).}\and E. Walsh\thanks{Simon Fraser University, Burnaby, BC, Canada, V5N IS6 (ewalsh@sfu.ca).}}

\begin{document}

\maketitle

\begin{abstract}
Many adaptive mesh methods explicitly or implicitly use equidistribution and alignment. These principles can be considered central to mesh adaption \cite{HRBook}. A Metric Tensor $\mathbf{M}$ is the tool by which one describes the desired level of mesh anisotropy. In contrast a mesh redistribution method based on the Monge-Amp\`ere equation \cite{BW:06}, \cite{BW:09}, \cite{BHRActa}, \cite{Budd96}, which combines equidistribution with optimal transport,  does not require the explicit construction of a Metric Tensor $\mathbf{M}$, although such an $\mathbf{M}$ always exists. An interesting question is whether such a method produces an anisotropic mesh. To answer this question we consider the general metric $\mathbf{M}$ to which an optimally transported mesh aligns. We derive the exact metric $\mathbf{M}$, involving expressions for its eigenvalues and eigenvectors, for a model linear feature. The eigenvectors of $\mathbf{M}$ are shown to be orthogonal and tangential to the feature, and the ratio of the eigenvalues is shown to depend, both locally and globally, on the value of the scalar density function $\rho=\sqrt{\det{\mathbf{M}}}$. We thereby demonstrate how an optimal transport method produces an anisotropic mesh along a given feature while equidistributing a suitably chosen scalar density function. Numerical results for a Parabolic Monge-Amp\`ere moving mesh method \cite{BW:06}, \cite{BW:09}, \cite{BHRActa}, \cite{Budd96}, \cite{walsh} are included to verify these results, and a number of additional questions are raised.
\end{abstract}



\pagestyle{myheadings}
\thispagestyle{plain}
\markboth{}{}
\section{Introduction}

\noindent Many non-linear partial differential equations (PDEs), including those for convection and reaction dominated problems, 
have solutions which exhibit a large variation  in a small region of the physical domain.
Numerical computations of such solutions often can be obtained more efficiently and accurately using some form of mesh adaptation/redistribution. 
For such methods it is usually desirable to adjust the size, shape and orientation  of the mesh elements to the underlying physical problem so that the mesh adapts to the geometry and flow field of the solution. Mesh adaption/redistribution has been applied in many areas of science and engineering, and has been used with great success to solve problems involving boundary layers, inversion layers, shock waves, ignition fronts, storm fronts, gas combustion and groundwater hydrodynamics \cite{walsh}, \cite{HZZ}, \cite{Hyman}, \cite{stockie}, \cite{T},  \cite{Tang},\cite{TangTang}. 

Problems often have solutions which display anisotropy, changing more significantly in one direction than the others, and
an anisotropic mesh is desirable to represent the solution features. However many adaptive methods, such as Winslow's celebrated method \cite{Wins}, explicitly adjust only the size of mesh elements, typically using the equidistribution of some measure of the solution as a guide, possibly enforcing unnecessary shape regularity. This can lead to isotropic meshes which are potentially inefficient for resolving the structure of the anisotropic solutions, in that they result in a large number of mesh points being concentrated along the anisotropic feature. Thus, there is considerable interest in finding moving mesh algorithms which 
can be assured to work well for anisotropic problems. The idea of using a Metric Tensor to quantify anisotropy was exploited in two-dimensional mesh generation as early as the 1990's \cite{Fortin}, \cite{Az} , and accurate a posteriori \cite{Pic}, \cite{H2010}, and a priori \cite{Form}, \cite{H2005}, anisotropic error estimates were developed. It was established that the absolute value of the Hessian matrix is a metric \cite{Hecht}, and since this metric arises in error bounds that estimate interpolation error, it can be used to generate a mesh which minimises interpolation error \cite{Hecht2}, \cite{Cao}, \cite{Vallet}, \cite{H2005b}.  

In this paper we consider an alternative approach in which an adapted mesh is generated by using an optimal transport procedure and solving an associated Monge-Amp\`ere equation. This method,
described in \cite{BW:06}, \cite{BW:09}, \cite{BHRActa}, \cite{walsh}, calculates a mesh which locally equidistributes a measure of the solution and satisfies certain global regularity constraints. It generates the mesh by solving a scalar equation, and has the advantages of being robust, flexible and cheap to implement, for both two and three dimensional problems. We shall show, both analytically and through numerical experiments,
that for anisotropic problems with strong linear features the enforcement of the global regularity conditions leads to anisotropic meshes closely aligned to these features and thus is well suited to PDE computations. In a forthcoming paper we will also show that these methods can also align to anisotropic features with strong curvature.

An outline of the paper is as follows. In Section 2 we consider the basic principles of equidistribution and alignment that are central to mesh adaptation. We then describe the 
above mesh generation method that combines equidistribution with ideas from optimal transport theory, and introduce alignment measures directly relevant to this method. In Section 3, for a set of model problems with strong linear features we rigorously derive the Metric Tensor $\mathbf{M}$ to which an optimally transported mesh aligns,
and thereby show close alignment to these features. In Section 4 we present numerical examples to verify the results in Section 3 and present further test cases to illustrate alignment properties for more complex linear and non-linear features. Lastly, our main conclusions are given in Section 5.

\section{Basic principles of anisotropic mesh redistribution}

\subsection{ Equidistribution and Alignment}
Following  \cite{HRBook}, an effective approach for studying redistribution of an initially uniform mesh is to generate an invertible coordinate transformation $\mathbf{x}=\mathbf{x}(\boldsymbol{\xi}):\Omega_c\rightarrow \Omega_p$, from a fixed computational domain $\Omega_c$ to the physical domain  $\Omega_p$ in which the underlying PDE is posed. The mesh in $\Omega_p$ is then generated as the image of a fixed uniform computational mesh in $\Omega_c$. The alignment and other features of the mesh can then be determined by calculating the properties of the transformation $\mathbf{x}(\boldsymbol{\xi})$.

Assuming for the moment that $\mathbf{x}$ and $\boldsymbol{\xi}$ are given, and for simplicity restricting attention to the  2D case, 
consider the local properties of this transformation. 
Let $\hat{K}$ be a circular set in $\Omega_c$, centred at $\boldsymbol{\xi_0}$, such that
$$\hat{K} = \{ \boldsymbol{\xi}: (\boldsymbol{\xi}-\boldsymbol{\xi_0})^{T}(\boldsymbol{\xi}-\boldsymbol{\xi_0})=\hat{r}^2 \},$$
where the radius $\hat{r}\propto(\vert \Omega_c \vert / N)^{1/2}$ and $N$ represents the number of mesh elements.
Linearizing  about  $\boldsymbol{\xi_0}$ we obtain
$$\mathbf{x}(\boldsymbol{\xi})=\mathbf{x}(\boldsymbol{\xi_0})+\mathbf{J}(\boldsymbol{\xi_0})(\boldsymbol{\xi}-\boldsymbol{\xi_0})+\mathrm{O}(\vert \boldsymbol{\xi}-\boldsymbol{\xi_0}\vert^2),$$ and
the corresponding image set $K=\mathbf{x}(\hat{K})$ in $\Omega_p$ is approximately given by
$${K} = \{ \mathbf{x}: (\mathbf{x}-\mathbf{x}(\boldsymbol{\xi_0}))^{T}\mathbf{J}^{-T}\mathbf{J}^{-1}(\mathbf{x}-\mathbf{x}(\boldsymbol{\xi_0}))=\hat{r}^2 \}.$$
As the set $K$ and $\boldsymbol{\xi_0}$ are arbitrary, we can replace $\boldsymbol{\xi_0}$ by a general point $\boldsymbol{\xi}$.
The Jacobian matrix $\mathbf{J}$ and its determinant $J$, referred to simply as the Jacobian, are 
\[
\mathbf{J}=\left[\begin{array}{cc} x_{\xi}& x_{\eta}\\ y_{\xi} & y_{\eta}\end{array}\right] \hspace{1cm} {J}=\left \vert \begin{array}{cc} x_{\xi}& x_{\eta}\\ y_{\xi} & y_{\eta}\end{array}\right \vert=x_{\xi}y_{\eta}-x_{\eta}y_{\xi}.
\]
Taking the singular value decomposition  
\[
\mathbf{J}=U\Sigma V^T, \hspace{1cm}\Sigma=\mathrm{diag}(\sigma_1,\sigma_2),
\]
it follows that
\[
{K} = \{ \mathbf{x}: (\mathbf{x}-\mathbf{x}(\boldsymbol{\xi_0}))^{T} \; U \; \Sigma^{-2} \;  U^T \; (\mathbf{x}-\mathbf{x}(\boldsymbol{\xi_0}))=\hat{r}^2 \}. 
\]
so that  the orientation of $K$ is determined by the left singular vectors  $U=[\mathbf{e_1},\mathbf{e_2}]$, and the size and shape by the singular values $\sigma_1$ and $\sigma_2$ (see Fig \ref{fig1:map}).
\begin{figure}[hhhhhhhh!!!!!]
\begin{center} \includegraphics[height=5cm,width=6cm]{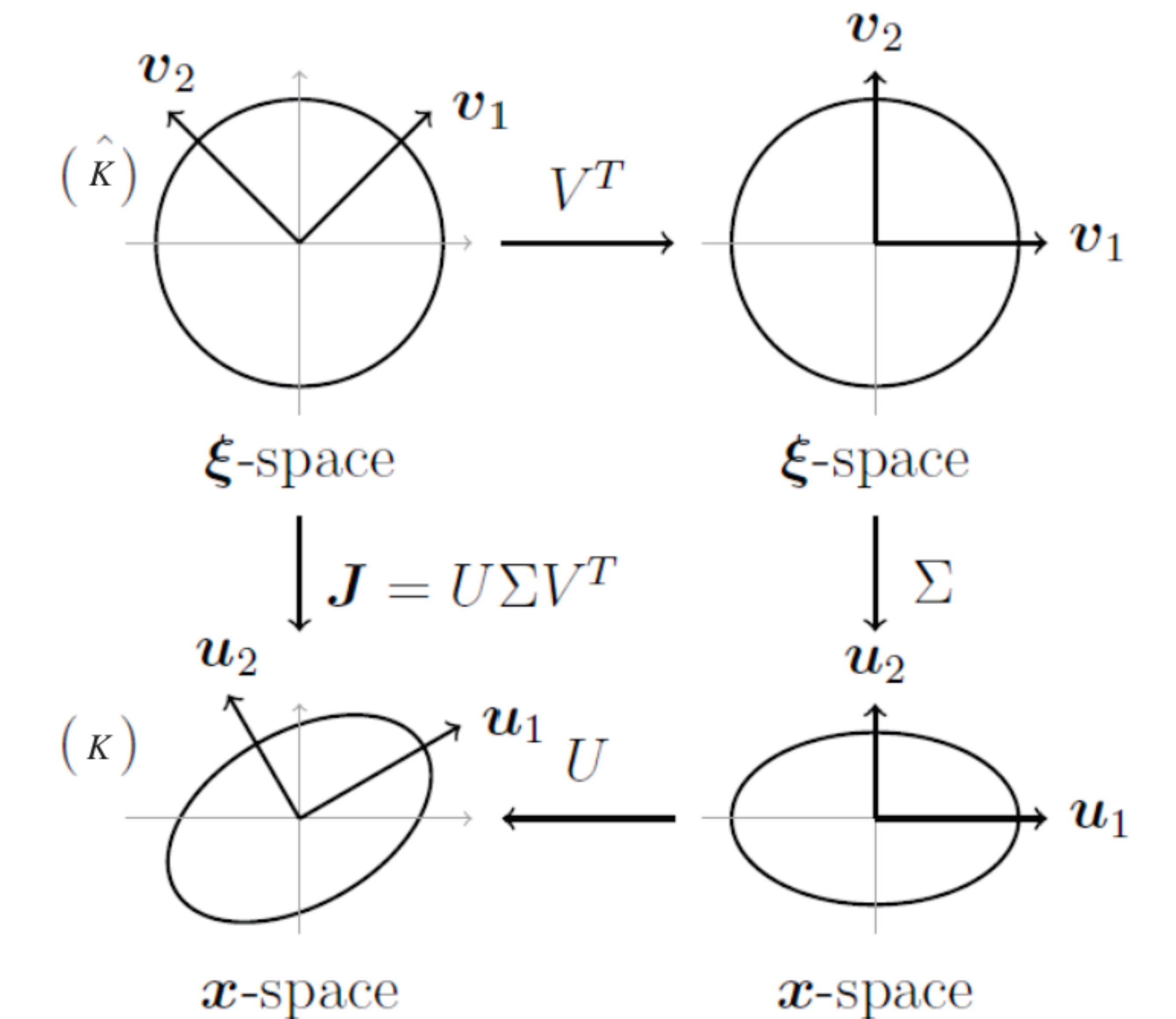}\end{center}
\caption{ \small The 2D mapping of a set  ($\hat{K}$, a circle) in $\Omega_c$, to a physical mesh element ($K$, an ellipse) in $\Omega_p$, under $\mathbf{x}(\boldsymbol{\xi})$. The local anisotropy of the transformation is evident from the degree of compression and stretching of the ellipse.}
\label{fig1:map}
\end{figure}

It immediately follows that we can quantify the size, shape and orientation of an element $K$, in the continuous sense, using the singular values and left singular vectors of $\mathbf{J}$, 
and the eigenvalues and eigenvectors of the associated {\em Metric Tensor} 
\begin{equation}
\mathcal{M}=\mathbf{J}^{-T}\mathbf{J}^{-1}. \label{eqal1}
\end{equation}
In particular, the eigenvectors of $\mathcal{M}$ are $\textbf{e}_1$,$ \textbf{e}_2$ and the eigenvalues $\mu_1$, $\mu_2$,  satisfy
$\mu_i=1/\sigma_{i}^{2}$ for $ i=1,2$ , with 
\begin{eqnarray}
\mathcal{M}=U\Sigma^{-2} U^T=\left[\begin{array}{cc} \mathbf{e_1}&  \mathbf{e_2}\end{array}\right]\left[\begin{array}{cc}\frac{1}{\sigma_1^2} & 0\\ 0 & \frac{1}{\sigma_2^2} \end{array}\right]\left[\begin{array}{c}  \mathbf{e_1}^T\\ \mathbf{e_2}^T\end{array}\right]. \nonumber
\end{eqnarray}
Hence, the circumscribed ellipse of a mesh element will have principal axes in the direction of the eigenvectors $\mathbf{e_1}$ and $\mathbf{e_2}$, with semi-lengths given by the  values $\sigma_1=\sqrt{1/\mu_1}$ and $\sigma_2=\sqrt{1/\mu_2}$, (although we note that in the discrete case the shape, size, and orientation of a mesh element are only partially determined by this metric). 
Accordingly, a useful measure $Q_s$ of the anisotropy of the mesh locally is given by the ratio of $\sigma_1$ and $\sigma_2$. A natural formulation in terms of $\mathbf{J}$ is \begin{equation}
Q_{s}=\frac{\mathrm{tr}(\mathbf{J^TJ})}{2\det(\mathbf{J^TJ})^{1/2}}=\frac{\sigma_1^2+\sigma_2^2}{2{\sigma_1 \sigma_2}}=\frac{1}{2}\left( {\frac{\sigma_1}{\sigma_2}}+  {\frac{\sigma_2}{\sigma_1}}\right). 
\label{qske}
\end{equation}
If the singular values are equal then the mesh is isotropic  and $Q_{s}=1$.
This measure, and the circumscribed ellipse of a mesh element, are extremely useful for visualising and analysing the degree of anisotropy, as we demonstrate later. We note that this is a local measure for anisotropy  \cite{HRBook}, and 
should be considered alongside more global measures of mesh quality such as the Kwok Chen metric \cite{KC00}. 

\subsection{Metric Tensors and locally M-Uniform meshes}
In this section we consider constructing locally anisotropic meshes using a Metric Tensor. If  we define a scaled Metric Tensor
\begin{equation}
\mathrm{\textbf{M}}= \theta \mathcal{M},  \label{eqal1b}
\end{equation}
for some constant $\theta$, then it follows from (\ref{eqal1}) that $\sqrt{\det(\mathrm{\textbf{M}})} J = \theta, $
for all $\mathbf{x}\in\Omega_p$. 
More specifically, if we define the {\em scalar density function} $\rho(\mathbf{x})=\sqrt{\det(\mathrm{\textbf{M}})}>0,$
then all elements have a constant area in the metric $\mathrm{\textbf{M}}$ such that 
\begin{equation}
\rho J=\theta.
\label{eqal2}
\end{equation}
Integrating the expression (\ref{eqal2}) over $\Omega_c$ and applying the change of variable formula it follows immediately that
\begin{equation}
\theta=\int_{\Omega_p}{\rho}  \; d\mathbf{x}/ \int_{\Omega_c}{d}\boldsymbol{\xi},\label{thetadef}
\end{equation}
or equivalently the mesh volume of  $\rho$ is equalised over each mesh cell. 
Equation (\ref{eqal2}) is the well known {\em equidistribution principle} which plays a fundamental role in mesh adaptation, giving direct control over the {\em size},
but not the alignment, of the mesh elements. 

For mesh generation in two or more dimensions the equidistribution principle (\ref{eqal2}) alone is insufficient to determine the  mesh uniquely and additional constraints are required \cite{Simp}.  Methods that augment the equidistribution principle with further local constraints are in
\cite{Baines}, \cite{Simp2}, \cite{H2001}, \cite{H2005}, \cite{Knupp}, and 
other principles for anisotropic mesh adaptation in\cite{steinberg}, \cite{BW:09}, \cite{Finn}.  A  common approach to locally controlled anisotropic mesh generation is to 
define the desired level of anisotropy through the Metric Tensor  $\mathrm{\textbf{M}}$ directly.
Then ${\mathbf M}$ is prescribed and the Jacobian ${\mathbf J}$ of the map is calculated {\bf directly} by enforcing the condition 
 \begin{equation}
Q_{a} \equiv \frac{\mathrm{tr}(\mathbf{J}^{T}\mathbf{M}\mathbf{J})}{2\det(\mathbf{J}^{T}\mathbf{M}\mathbf{J})^{1/2}} = 1.\label{eqal3}
\end{equation}
This extends the anisotropy measure (\ref{qske}) and requires that now
all elements are equilateral with respect to the metric $\mathrm{\textbf{M}}$. This condition allows us to directly control the shape and orientation of a mesh element through an appropriate choice of $\mathrm{\textbf{M}}$. It also follows directly from (\ref{eqal1}) and (\ref{eqal1b}), and is referred to  as the {\em alignment condition} \cite{HRBook}. Huang \cite{H2001} shows that combining the equidistribution and alignment conditions (\ref{eqal2})-(\ref{eqal3}) gives 
\begin{equation}
\mathbf{J}^{-T}\mathbf{J}^{-1}=\theta^{-1}\mathrm{\textbf{M}},\hspace{.2cm}\mbox{or equivalently}\hspace{.2cm}\mathbf{J}^{T}\mathrm{\textbf{M}}\mathbf{J}=\theta I.\label{eqal4a}
\end{equation}
That is, when the coordinate transformation satisfies relation (\ref{eqal4a}), the element size, shape, and orientation are completely determined by $\mathrm{\textbf{M}}$ throughout the domain. The resulting mesh will be aligned to the metric $\mathbf{M}$ and equidistributed with respect to the measure $\rho$, and is referred to as {\em M-uniform} \cite{HRBook}. In general there is no unique solution to (\ref{eqal4a}), and so in practice this condition can only be enforced approximately. The choice of an appropriate Metric Tensor is important to the success of this method, and typically those which lead to low interpolation 
errors are chosen.

The simplest choice is to take a matrix monitor function of the form 
\begin{equation}
\mathrm{\mathbf{M}}=\rho I.\label{smm}
\end{equation}
Using a variational approach this is equivalent to Winslow's variable diffusion method \cite{Wins}. 
In this case, by condition (\ref{eqal4a}), $\mathbf{J^{-T} J^{-1}}$ is a scalar matrix. This  means 
the singular values, and hence the semi-lengths of the circumscribed ellipse of a mesh element, are equal (i.e., it is a circle) if
(\ref{smm}) is exactly satisfied.

In contrast, Huang \cite{H2005} has derived the exact forms of $\mathrm{\textbf{M}}$ for which the resulting mesh minimizes the interpolation error of some underlying function $u$. Piecewise constant interpolation error can be minimised in the L-2 norm if 
\begin{equation}
\mathrm{\textbf{M}}=\kappa_{h,1}[I+\alpha_{h,1}^2 \nabla u\nabla u^T]\label{eqal5}
\end{equation}
where 
$$\alpha_{h,1}=\left({\beta^{-1}}{(1-\beta)\int_{\Omega_p}\| \nabla u\|^{1/2} d\mathbf{x}}\right)^2, \quad \kappa_{h,1}=(1+\alpha_{h,1}^2\nabla u\nabla u^T)^{-1/4},$$
and $\beta$ is a parameter which controls the percentage of mesh points that are concentrated in regions where $\rho$ is large. For piecewise linear interpolation, the optimal
Metric Tensor is given by
\begin{equation}
\mathrm{\textbf{M}}=\kappa_{h,2}[I+\alpha_{h,2}\vert H(u)\vert],\label{eqal6}
\end{equation}
for suitable parameters $\kappa_{h,2}$, and $\alpha_{h,2}$, where $H(u)$ is the Hessian matrix of $u$.

Whilst effective in generating (essentially optimal) anisotropic meshes, these methods require finding the full Jacobian of the map at each step, which necessitates incorporating
extra convexity conditions to ensure uniqueness, making the resulting (typically variational) methods
challenging to implement. In contrast Winslow's method is rather simpler to use. However, such methods that use a  scalar matrix monitor function may well be too restrictive to produce a mesh that is aligned to a physical solution \cite{HRBook}. This begs the question of whether a method that equidistributes a scalar mesh density function is generally capable of producing anisotropic meshes. We demonstrate in the next section that by  combining equidistribution of a \emph{scalar density function} with a {\em global} constraint, namely \emph{optimal transport}, we can produce anisotropic meshes which are relatively easy to compute.
Furthermore, for certain linear features, we are able to derive analytically the precise form of the metric $\mathbf{M}$ to which these meshes align and show it has a similar form to those metrics  given in (\ref{eqal5}) and (\ref{eqal6}) which minimise interpolation error. 

\subsection{Metric Tensors for Mesh generation} We conclude this section with some further observations on general Metric Tensors.
Any symmetric, positive definite Metric Tensor $\mathbf{M}$ with normalised orthogonal eigenvectors $\mathbf{e_1}$ and $\mathbf{e_2}$ and associated
eigenvalues $\mu_1, \mu_2$ can be expressed as
$$
\mathbf{M}= \mu_1\mathbf{e_1}\mathbf{e_1}^T+ \mu_2 \mathbf{e_2} \mathbf{e_2}^T.
$$
Since the identity matrix 
$I=\mathbf{e_1}\mathbf{e_1}^T+ \mathbf{e_2} \mathbf{e_2}^T,$
it follows that 
$$
\mathbf{M}=\mu_2 [I+(\mu_1/\mu_2 -1) \mathbf{e}_1 \mathbf{e_1}^T].
$$
We now consider some special cases. If $\mu_1=\mu_2=\rho(\mathbf{x})$ then
we obtain the scalar matrix valued monitor function (\ref{smm}).
If instead $\mu_2=1/ \mu_1$ then
\begin{eqnarray}
\mathbf{M}={\mu_1}^{-1}[I+(\mu_1^2-1) \mathbf{e_1}\mathbf{e_1}^T]. 
\end{eqnarray}
With a variational approach, this is equivalent to the method based upon harmonic maps \cite{HRBook}.  
The mesh adaptation occurs mainly in the directions of $\mathbf{e_1}$ and
its orthogonal complement. For an underlying function $u(x,y)$ if we define $\mathbf{e}_1= \nabla u/ \| \nabla u\|$, $\mu_1=\sqrt{1+\| \nabla u\|^2}$, and $\mu_2=1/\mu_1$,
then
\[
\mathbf{M}=({\sqrt{1+\| \nabla u\|^2}})^{-1} [I+ \nabla u \nabla u ^T].
\]
For problems in which $u$ has steep fronts or even discontinuities then coordinate line
compression and expansion occur mainly in the gradient direction, since $\mu_1$ and $\mu_2$ change much faster in the gradient direction than the
tangential direction. If there is no change in the gradient of the solution along the tangential direction then we may choose $\mu_2=1$, in which case we obtain the arc-length like matrix monitor function
$
\mathbf{M}= [I+ \nabla u \nabla u ^T]^{1/2},
$
which is a rescaling of the expression (\ref{eqal5}).
Dvinsky \cite{Dvi} uses the more general metric
\begin{eqnarray}
\mathbf{M}&=&[I+f(\Psi){\nabla \Psi \nabla \Psi^{T}}/{\| \nabla \Psi \|^2}],\label{dvinsky}
\end{eqnarray}
to obtain a mesh with good alignment and concentration around a curve defined by $\Psi(\mathbf{x})=0$, where $f(\Psi)$ is a function of the distance from a given point to this curve that increases as the distance tends to zero. The eigenvalues of $\mathbf{M}$ are then $\mu_1=1+f(\Psi)$ and $\mu_2=1$, and the corresponding eigenvectors are $\mathbf{e}_1={\nabla \Psi}/{\| \nabla \Psi \|}$ and it's orthogonal complement. Therefore mesh cells have a shorter length in the $\nabla \Psi$ direction whenever $f(\Psi)$ changes rapidly.  Since the density function $\rho=\sqrt{\det{\mathbf{M}}}=\sqrt{1+f(\Psi)},$
it follows that $\mu_1=\rho^2$ and $\mu_2=1$, and we can rewrite (\ref{dvinsky}) as 
\begin{eqnarray}
\mathbf{M}&=&[I+(\rho^2-1){\nabla \Psi \nabla \Psi^{T}}/{\| \nabla \Psi \|^2}].\label{dvinsky2}
\end{eqnarray}
If  $\Psi(\mathbf{x})=u(\mathbf{x})$, and $f(\Psi)=\| \nabla u \|^2$ (which corresponds to $\rho=\sqrt{1+\| \nabla u \|^2}$),  then (\ref{dvinsky}) and (\ref{dvinsky2}) are equivalent to the matrix monitor function 
$
\mathbf{M}=[I+ \nabla u \nabla u ^T].
$
\section{Mesh redistribution using global constraints and the Monge-Amp\`ere equation} As stated in Section 2, the local scalar equidistribution condition (\ref{eqal2}) does not uniquely define a mesh and needs to be augmented by additional constraints.
In contrast to the last section, we consider here the use of {\em global constraints} to define the mesh, viz., we use {\em Optimal Transport Regularisation}.
Instead of enforcing local structure, we seek to find a mesh
as close as possible (in a suitable norm) to a uniform one, consistent with satisfying (\ref{eqal2}).
In this section we describe such a method, show how to calculate the associated Metric Tensor ${\mathbf M}$, and subsequently examine their excellent
alignment properties.

We consider the mesh as defined in terms of an appropriate map (representing the limiting case as the mesh is refined) and use a global constraint requiring this map
to be close to the identity.
\begin{definition}
An optimally equidistributed mapping $\mathbf{x}({\mathbf \xi})$ is one which minimizes the functional $I_{2}$, where
\[
I_2 =\int_{\Omega_c}\vert \mathbf{x} (\boldsymbol{\xi}) - \boldsymbol{\xi}\vert^2d\mathbf{x},
\]
over all invertible $\mathbf{x} (\boldsymbol{\xi})$ for which the equidistribution condition (\ref{eqal6}) also holds.
\end{definition}\\
The following result gives both the existence and uniqueness of such a map and a means to calculate it.
\begin{theorem} (Brenier, Cafferelli) There exists a unique optimal
mapping $\mathbf{x}(\boldsymbol{\xi})$ satisfying the equidistribution condition (\ref{eqal2}). This map has the same regularity as $\rho$. Furthermore,
the map $\mathbf{x}(\boldsymbol{\xi})$ can be written as the gradient
(with respect to $\boldsymbol{\xi}$) of a unique (up to constants) convex mesh potential $P(\boldsymbol{\xi}, t)$, so that
\[
\mathbf{x}( \boldsymbol{\xi}) = \nabla_{\xi}P(\boldsymbol{\xi}), \hspace{1cm} \Delta_{\xi}P(\boldsymbol{\xi}) > 0.
\]
\end{theorem}
It is immediate that if $\mathbf{x} = \nabla_{\xi}P$ then the Jacobian matrix ${\mathbf J}$ is {\em symmetric} and is the Hessian matrix of $P$, i.e.,
\begin{eqnarray}
\mathbf{J}=\mathbf{J}^T=\left[\begin{array}{cc} x_{\xi}& x_{\eta}\\ y_{\xi} & y_{\eta}\end{array}\right] =\left[\begin{array}{cc} P_{\xi\xi}& P_{\xi\eta}\\ P_{\eta\xi} & P_{\eta\eta}\end{array}\right]=:\mathbf{H}(P). \nonumber
\end{eqnarray}
Furthermore, the Jacobian determinant $ J$ is the Hessian determinant of $P$ such that
\[
{J}=x_{\xi}y_{\eta}-x_{\eta}y_{\xi}=P_{\xi\xi}P_{\eta\eta}-P_{\xi\eta}^2:=H(P).
\]
The equidistribution condition (\ref{eqal2}) thus becomes 
\begin{equation}
\rho(\nabla P)  H(P) =\theta, \label{pmaeqa1}
\end{equation}
which is the  {\em Monge-Amp\`ere equation}.  This fully nonlinear equation is generally augmented with Neumann or periodic boundary conditions, where the boundary of $\Omega_c$
is mapped to the boundary of $\Omega_p$.
The gradient of $P$  thereby gives the unique map $\mathbf{x}$.
Methods to solve (\ref{pmaeqa1}) are described in \cite{BW:09},\cite{Finn}, and form the basis of effective and robust mesh redistribution algorithms in two and three dimensions \cite{Pab}.
These methods have several advantages in practical applications. In particular, they only involve solving scalar equations, they deal naturally with
complex boundaries, and they can be easily coupled to existing software for solving PDEs \cite{walsh} and also for operational data assimilation  \cite{culpic}.

Regions where $\rho$ is large will result in small mesh elements and vice versa. However, it is not immediately clear what 
control one has through the choice of $\rho$ over the shape and orientation of the elements. 
To study this we use the following result in which we uniquely determine the Metric Tensor of the resulting map.
\begin{lemma} For a given $\rho(\mathbf{x})$ and solution of (\ref{pmaeqa1}), the corresponding mesh determines a unique Metric Tensor 
${\mathbf M}$, for which $\rho = \sqrt{\det({\mathbf M})}$.\end{lemma}
 \begin{proof} Given $\rho(\mathbf{x})$, the Monge-Amp\'ere equation (\ref{pmaeqa1}) has a unique solution $P$. Hence we may uniquely construct
the Jacobian matrix  ${\mathbf J} = {\mathbf H}(P)$ and Metric Tensor ${\mathbf M} = \theta {\mathbf J}^{-1} {\mathbf J}^{-T}.$ Since 
$J \sqrt{\det(M)} = \theta = \rho J$ from (\ref{pmaeqa1}), the result follows. \qquad\end{proof}

We can calculate the explicit form of ${\mathbf M}$ as follows: Since ${\mathbf J}$ is {\em symmetric} its eigenvalues $\lambda_1,\lambda_2$ are equal to its singular values 
$\sigma_1,\sigma_2$ and
its (unit) eigenvectors ${\mathbf e}_1$ and ${\mathbf e}_2$ are orthogonal. It can therefore be expressed in the form
$$
{\mathbf J} = \lambda_1 {\mathbf e}_1 {\mathbf e}_1^T + \lambda_2 {\mathbf e}_2 {\mathbf e}_2 ^T,
$$
implying $\rho={\theta}/{J}={\theta}/{\lambda_1 \lambda_2}.$ It follows from (\ref{eqal4a}) that  the Metric Tensor ${\mathbf M}$ for which the mesh is M-uniform has the same (unit) orthogonal eigenvectors ${\mathbf e}_1$ and ${\mathbf e}_2$,
and eigenvalues $\mu_1 = {\theta}/{\lambda_1^2}, \mu_2 = {\theta}/{\lambda_2^2}$
and can be expressed in the form 
\begin{equation}
{\mathbf M} = \theta \left( \lambda_1^{-2}  {\mathbf e}_1 {\mathbf e}_1^T + \lambda_2^{-2}  {\mathbf e}_2 {\mathbf e}_2 ^T \right).
\label{cbjac2}
\end{equation}
Observe that this Metric Tensor is not (necessarily) a scalar multiple of the identity matrix. In Section 4 we study the (local) alignment properties of this Metric Tensor
determined by the global constraint of minimising $I_2$.

\section{Alignment to a linear feature}
\subsection{Construction of an exact map} If the scalar density $\rho(\mathbf{x})$ is concentrated along
linear features, characterised by orthogonal vectors ${\mathbf e}_1$ and ${\mathbf e}_2$, the Monge Ampere equation can be solved exactly. For the 
resulting mapping we shall show that the Metric Tensor ${\mathbf M}$ satisfies
(\ref{cbjac2}), implying the mesh aligns along the linear features. 

Consider the case where $\Omega_c = \Omega_p = [0,1]^2$ and the solution to (\ref{pmaeqa1}) 
is a {\em doubly-periodic} map from $\Omega_c\rightarrow\Omega_p$, such that $\boldsymbol{\xi}=[\xi,\eta]\in\Omega_c$, $\mathbf{x}=[x, y]\in\Omega_p$. To characterise linear features we introduce orthogonal unit vectors ${\mathbf e}_1$ and ${\mathbf e}_2$
and consider a doubly-periodic (in ${\mathbf x}$) scalar density function of the form
\begin{eqnarray}
\rho(\mathbf{x})=\rho_1(\mathbf{x}\cdot \mathbf{e_1})\rho_2(\mathbf{x}\cdot \mathbf{e_2}):=\rho_1(x^{\prime})\rho_2(y^{\prime})\nonumber
\end{eqnarray}
where $\mathbf{e_1}=\left[\begin{array}{c}a\\ b\end{array}\right],
\mathbf{e_2}=\left[\begin{array}{c}-b\\ a\end{array}\right], \quad a^2 + b^2 = 1.$
Assume furthermore that the periodic function $\rho_1$ is large when $\mathbf{x}\cdot \mathbf{e_1}=c$, and the periodic function $\rho_2$ is large when $\mathbf{x}\cdot \mathbf{e_2}=d$,  for given constants $c$, and  $d$, and they are
small (close  to $1$) otherwise. Our motivation for this choice of density function $\rho$ is that the solution of the equidistribution equation (\ref{eqal1}) would be expected to concentrate mesh points 
along the lines given by either of the conditions  $\mathbf{x}\cdot \mathbf{e_1}=c$, or $\mathbf{x}\cdot \mathbf{e_2}=d$. To deduce properties of the 
mesh alignment we must solve the full Monge-Amp\`ere equation, with $\theta$ calculated as below.

 \begin{lemma} The value of $\theta$ is given by
$$
\theta = \theta_1 \theta_2, \quad \mbox{where} \quad \theta_1 = \int_0^1 \rho_1(s) \; ds, \quad \mbox{and} \quad \theta_2 = \int_0^1 \rho_2(s) \; ds. 
$$
\end{lemma}
 \begin{proof} By the definition in expression (\ref{thetadef}) 
$$\theta = \int_{\Omega_p} \rho({\mathbf x})  \; dx/\int_{\Omega_c} \; d\xi = \int_0^1 \int_0^1 \rho_1({\mathbf x}\cdot{\mathbf e}_1) \rho_2({\mathbf x}\cdot{\mathbf e}_2) \; dx dy/ \int_0^1 \int_0^1 \; d\xi d\eta.$$
We now introduce new coordinates $x'$ and $y'$ given by $x' = {\mathbf x}\cdot{\mathbf e}_1, \quad y'={\mathbf x}\cdot{\mathbf e}_2.$
As  $\mathbf{e_1}$ and $\mathbf{e_2}$ are orthonormal it follows immediately that $dx \; dy = dx' \; dy'$, so exploiting the double-periodicity of the function $\rho$  we may rewrite the integral as
$$\theta = \int_0^1 \int_0^1 \rho_1(x') \rho_2(y') \; dx' dy' =  \int_0^1 \rho_1(x') \; dx'  \int_0^1 \rho_2(y') \; dy' = \theta_1 \theta_2.$$ \qquad\end{proof}

Note that the Monge-Amp\`ere equation (\ref{pmaeqa1}) can be expressed in the form
\begin{eqnarray}
H(P) \;  \rho_1(x^{\prime}) \rho_2(y^{\prime}) = \theta_1 \theta_2. \label{MAls}
\end{eqnarray}
Remarkably, this fully nonlinear PDE is separable and has an exact solution.
\begin{lemma}For appropriate periodic functions $F(t)$ and $G(t)$ there exists a doubly-periodic, separable solution to (\ref{MAls}) of the form
\begin{equation}
P(\xi,\eta)=F(\boldsymbol{\xi}\cdot\mathbf{e_1})+G(\boldsymbol{\xi}\cdot\mathbf{e_2}).\label{VSS}
\end{equation}
Furthermore, this solution is unique up to an arbitrary constant of addition.
\end{lemma}
\begin{proof} Differentiating (\ref{VSS}) with respect to $\xi$ and $\eta$ gives
\begin{eqnarray}
\mathbf{x}=\nabla_{\xi}P=\mathbf{e_1}^TF^{\prime}+\mathbf{e_2}^TG^{\prime}.\label{mapls}
\end{eqnarray}
Differentiating again with respect to $\xi$ and $\eta$ we obtain
\begin{equation}
P_{\xi\xi} = a^2 F^{\prime\prime}+b^2 G^{\prime\prime}, \quad 
P_{\xi\eta } = ab F^{\prime\prime}-ab G^{\prime\prime}, \quad
P_{\eta\eta} = b^2 F^{\prime\prime}+a^2 G^{\prime\prime}.\nonumber
\end{equation}
Hence 
\[
\mathbf{H}(P) =\left[\begin{array}{cc}\mathbf{e_1}&\mathbf{e_2} \end{array}\right]\left[\begin{array}{cc}F^{\prime\prime}&0 \\ 0&G^{\prime\prime}\end{array}\right]\left[\begin{array}{c}\mathbf{e_1}^T\\ \mathbf{e_2}^T\end{array}\right]
\]
and
\begin{eqnarray}
H(P)&=&(a^2F^{\prime\prime}+b^2G^{\prime\prime})(b^2F^{\prime\prime}+a^2G^{\prime\prime})-(abF^{\prime\prime}-abG^{\prime\prime})^2\nonumber\\
&=&(b^2+a^2)^2F^{\prime\prime}G^{\prime\prime}=F^{\prime\prime}G^{\prime\prime}.\label{HessVSS}
\end{eqnarray}
Substituting (\ref{HessVSS}) into the Monge Ampere equation (\ref{MAls}) we obtain
\[
F^{\prime\prime}(\xi^{\prime})G^{\prime\prime}(\eta^{\prime}) \; \rho_1(x') \rho_2(y') = \theta_1 \theta_2,
\]
where $\xi' = \boldsymbol{\xi}\cdot\mathbf{e_1}$ and $\eta' = \boldsymbol{\xi}\cdot\mathbf{e_2}.$
Now by (\ref{mapls}) it follows that
\[
x^{\prime}=\mathbf{x}\cdot \mathbf{e_1}=\mathbf{e_1}^T\cdot \mathbf{e_1}F^{\prime}+\mathbf{e_2}^T\cdot \mathbf{e_1}G^{\prime}=F^{\prime}(\xi^{\prime}),
y^{\prime}=\mathbf{x}\cdot \mathbf{e_2}=\mathbf{e_1}^T\cdot \mathbf{e_2}F^{\prime}+\mathbf{e_2}^T\cdot \mathbf{e_2}G^{\prime}=G^{\prime}(\eta^{\prime}).
\]
Thus, there is a solution of (\ref{MAls}) of the form (\ref{VSS}) provided $F$ and $G$ satisfy 
\begin{equation}
{
F^{\prime\prime}(\xi^{\prime}) \rho_1(F'(\xi')) =\theta_1 \alpha \quad \mbox{and} \quad  G^{\prime\prime}(\eta^{\prime}) \rho_2(G'(\eta')) =\theta_2/\alpha, 
}
\label{EVfg}
\end{equation}
where $\alpha$ is (at this stage) an arbitrary constant. From the identities $x' = F'$ and $y' = G'$ it follows that 
$x'(\xi') \rho_1(x'(\xi')) = \theta_1 \alpha$ and for a suitable constant $c_1$, $R_1(x') \equiv \int_0^{x'}  \rho_1(s) \; ds = \theta_1 \alpha \;  \xi' + c_1.$
Since the map from $\Omega_c$ to $\Omega_p$  is doubly periodic, 
$x'(0) = 0$ and $x'(1) = 1$. Thus, $c_1 = 0$ and from the definition of $\theta_1$, $\alpha = 1$. 
Hence, we have
\begin{equation}
{
x' = {\mathbf x}\cdot{\mathbf e}_1 = R_1^{-1}(\theta_1 \; \xi') = R_1^{-1}(\theta_1 \; {\boldsymbol{\xi}}\cdot{\mathbf e}_1).
}
\label{xstuff}
\end{equation}
A similar identity follows for  $y'$ with related function $R_2$ and constant $c_2$, giving
\begin{equation}
{
 y' = {\mathbf x}\cdot{\mathbf e}_2 = R_2^{-1}(\theta_2 \; \eta') = R_2^{-1}(\theta_2 \; {\boldsymbol{\xi}}\cdot{\mathbf e}_2).
 }
 \label{ystuff}
\end{equation}
These define the functions $F$ and $G$, and
the uniqueness (\ref{VSS}) follows from the uniqueness of solutions of the Monge-Amp\`ere equation (\ref{MAls}) 
with periodic boundary conditions \cite{Loeper}. \qquad\end{proof}

Having found the solution of the Monge-Amp\`ere equation, we can now calculate the Jacobian of the map ${\mathbf J}$ and the Metric Tensor ${\mathbf M}$.  From the above 
$$
\mathbf{x}=\nabla_{\xi}P=\mathbf{e_1}^T R_1^{-1}(\theta_1 \xi^{\prime})+\mathbf{e_2}^T R_2^{-1}(\theta_2 \eta^{\prime})
$$
and the Jacobian matrix has the form
\begin{equation}
{
{\mathbf J} = \frac{\theta_1}{\rho_1(F^\prime(\xi^\prime))}\; {\mathbf e_1}\;{\mathbf e_1^T} +\frac{\theta_2}{\rho_2(G^\prime(\eta^\prime))} \;{\mathbf e_2}\;{\mathbf e_2^T}
}
\label{jacex34}
\end{equation}
with eigen/singular values 
\begin{equation}
\lambda_1 =  \theta_1/\rho_1, \hspace{.2cm}\mbox{and}\hspace{.2cm} \lambda_2=\theta_2/\rho_2.
\label{spotty1}
\end{equation}
From (\ref{eqal4a}), we infer that the mesh will be aligned to the metric
\begin{equation}
{\mathbf M} = \frac{\theta_2  \rho_1^2 }{\theta_1} \; {\mathbf e_1}\;{\mathbf e_1^T} +\frac{\theta_1 \rho_2^2}{\theta_2} \;{\mathbf e_2}\;{\mathbf e_2^T},
\label{jacex35c}
\end{equation}
with eigenvalues
\begin{equation}
\mu_1 ={\theta_2 \rho_1^2}/{\theta_1} \quad \mbox{and} \quad \mu_2 ={\theta_1 \rho_2^2}/{\theta_2}.
\label{spotty2}
\end{equation}
This Metric Tensor can be expressed in the equivalent form
\begin{eqnarray}
\mathbf{M}&=&\frac{\theta_1\rho_2^2}{\theta_2}[I+\frac{\theta_2^2\rho_1^2}{\theta_1^2\rho_2^2}-1]\mathbf{e_1}\mathbf{e_1}^T].\label{metex3}
\end{eqnarray}

These explicit forms for ${\mathbf J}$ and ${\mathbf M}$ reveal the alignment properties
of the map. Specifically, the eigendecomposition of ${\mathbf J}$ in (\ref{jacex34})
shows that the semi-axes of the ellipses described in Section 2  are parallel to
${\mathbf e_1}$ and ${\mathbf e_2}$ and thus align with the linear features.
The linear features we are aiming to represent arise when ${\mathbf x}\cdot{\mathbf e_1} = x' =  c$ and ${\mathbf x}\cdot{\mathbf e_2} = y' =  d$ so that 
respectively either $\rho_1$ is large and $\rho_2$ is not, or $\rho_2$  is large
and $\rho_1$ is not.
The anisotropy measure (\ref{qske}) in this case is then given by
$$Q_s = \frac{1}{2} \left( \frac{\theta_1 \rho_2}{\theta_2 \rho_1} + \frac{\theta_2 \rho_1}{\theta_1 \rho_2} \right) $$
which is clearly related to the relative size of the density functions $\rho_1$ and $\rho_2$, both locally and globally.  Along the linear features, where either $\rho_1 \gg1$ and $\rho_2={\cal O}(1)$, or $\rho_2 \gg 1$ and $\rho_1={\cal O}(1)$, the mesh elements will be anisotropic. Away from the linear feature, where $\rho_1$ and $\rho_2$ are both of order one, the degree of anisotropy is determined from the relative values of the density functions in the entire domain, $\theta_1$ and $\theta_2$.
\subsection{A single linear feature and the related Metric Tensor}
We now consider the special case of a periodic set of parallel (single) linear features  at an arbitrary angle relative to the coordinate axes.
For this we take $\rho_2 = \theta_2=1$ and
$$
\rho(\mathbf{x})=\rho_1(x^\prime), \quad \theta = \theta_1.
$$
In this special case $G^{\prime\prime}=1$, and so $G^{\prime}=\eta^{\prime}$,
$\mathbf{x}=\mathbf{e_1}^T R_1^{-1}(\theta\xi^{\prime})+\mathbf{e_2}^T \eta^{\prime}$, and
$$
{\mathbf J} =  \frac{\theta} {\rho}\; {\mathbf e_1}\;{\mathbf e_1^T} + {\mathbf e_2}\;{\mathbf e_2^T},
$$
so the associated Metric Tensor is
\begin{eqnarray}
\mathbf{M}&=&\theta[I+(\frac{\rho^2}{\theta^2}-1) \; \mathbf{e_1}\mathbf{e_1}^T].\label{metex1}
\end{eqnarray}

The cases presented thus far are for a prescribed density function. However, in a typical calculation,  the density function $\rho$ is based on some underlying function $u(\mathbf{x})$
that we seek to approximate on the mesh. It is therefore useful to consider the Metric Tensor in terms of this function. This will be especially instructive if we are to draw a meaningful comparison between the metrics derived here and the standard ones used by variational methods, especially those known to minimise interpolation error of $u({\mathbf x})$  which are given in (\ref{eqal5}) and (\ref{eqal6}). 

For this special case with a strongly anisotropic function $u({\mathbf x}) \equiv u({\mathbf x}\cdot{\mathbf e_1}) = u(x')$,
$$\nabla u = {\mathbf e}_1 u' \quad \mbox{and} \quad \|\nabla u\|^2 = (u')^2.$$
A commonly used scalar  metric is the arc-length density function 
\[
\rho({\mathbf x})=\sqrt{1+\alpha_h\| \nabla u\|^2}
\]
which for this anisotropic function is simply 
$$\rho({\mathbf x}) = \sqrt{1 + \alpha_h (u'(x'))^2 }, $$
and thus has precisely the form considered at the start of this section. It follows from (\ref{metex1}) that the associated Metric Tensor is 
\[
\mathbf{M}=\theta[I +\alpha\nabla u\nabla u^T],\hspace{.25cm}\mbox{where}\hspace{.25cm}\alpha={(1+\alpha_h\| \nabla u\|^2-\theta^2)}/{\theta^2\| \nabla u\|^2}.
\]
This Metric Tensor is very similar in structure to those typically used when constructing a mesh directly with a metric based approach. 
It has the same form as the metric in (\ref{eqal5}), a metric known to minimise piecewise constant interpolation error. However, there are subtle differences between the two. Notice that $\alpha$ is not constant here as in (\ref{eqal5}). 
Also, it is not possible to select a density function for which $\alpha$ is a constant when generating an optimally transported mesh.

Choosing instead the density function  
\[
\rho=\sqrt{1+\alpha_h(\vert u_{xx}\vert+\vert u_{yy}\vert)},
\]
gives
$$\mathbf{M}=\theta[I+\alpha \vert H(u)\vert] \quad \mbox{where} \quad \alpha=\frac{1+\alpha_h(\vert u_{xx}\vert+\vert u_{yy}\vert)-\theta^2}{\theta^2(\vert u_{xx}\vert+\vert u_{yy}\vert)}.$$
This metric has the same structural form as (\ref{eqal6}), a metric known to minimise piecewise linear interpolation error. However, there are again
subtle differences between the two, since $\alpha$ is not constant whereas $\theta$ is. The implications of these differences with regard to error minimisation requires further investigation and is the subject of a forthcoming paper.
\subsection{Examples}
We now consider two specific analytical examples which illustrate the theory described above.
\subsubsection{Example 1: A single periodic shock}
As a first example we consider a periodic array of linear features aligned at $\pi/4$ to the coordinate axes so that ${\mathbf e_1}^T=(1 \; 1)/\sqrt{2}$
and ${\mathbf e}_2^T=(1 \; -1)/\sqrt{2}.$  As a periodic mesh density we take
$$
\rho(\mathbf{x})=1+50\sum_{n=-\infty}^{\infty} \mathrm{sech}^2(50(\sqrt{2}x^{\prime}-n)):=\rho_1(x^{\prime}), \quad x' = {\mathbf x}\cdot{\mathbf e_1}.
$$
This density is concentrated along a set of lines of width $1/50\sqrt{2}$ which are parallel to ${\mathbf e}_2$, one of which passes through the coordinate origin,
and the others arising when $x'=\pm 1/\sqrt{2}, \pm 2/\sqrt{2}, \ldots$. Note that along each such line $\rho = 51 + {\cal O}(\exp(-50))$ and away from each such line  $\rho = 1 + {\cal O}(\exp(-50)).$

A direct calculation gives
\begin{equation}
{
\theta=\theta_1=\int_{\Omega_p}\rho(\mathbf{x}) \;  d\mathbf{x}= 3 + {\cal O}(\exp(-50)).
}
\label{thec}
\end{equation}
Similarly
\[
R_1(x^{\prime})=x^{\prime}+\frac{1}{\sqrt{2}}\sum_{n=-\infty}^{\infty}[\tanh(50(\sqrt{2}x^{\prime}-n))-\tanh(-50n)].
\]
The inverse of $R_1$ can be computed by fitting a spline through the data points $(R_1(x^{\prime}_i), \hspace{.2cm}x^{\prime}_i)$, for $x^{\prime}_i=\sqrt{2}i/N'$, $i=0,...,N'$. A plot of $R_1^{-1}$ is given in Fig. \ref{R1} for $N'=1000$. 
\begin{figure}[hhhhhhhhhhhhhhhhhhhhhhhhhhhhhhhhhhhhhhhhhhhhhh]
\begin{center}\includegraphics[height=5.5cm,width=6.5cm]{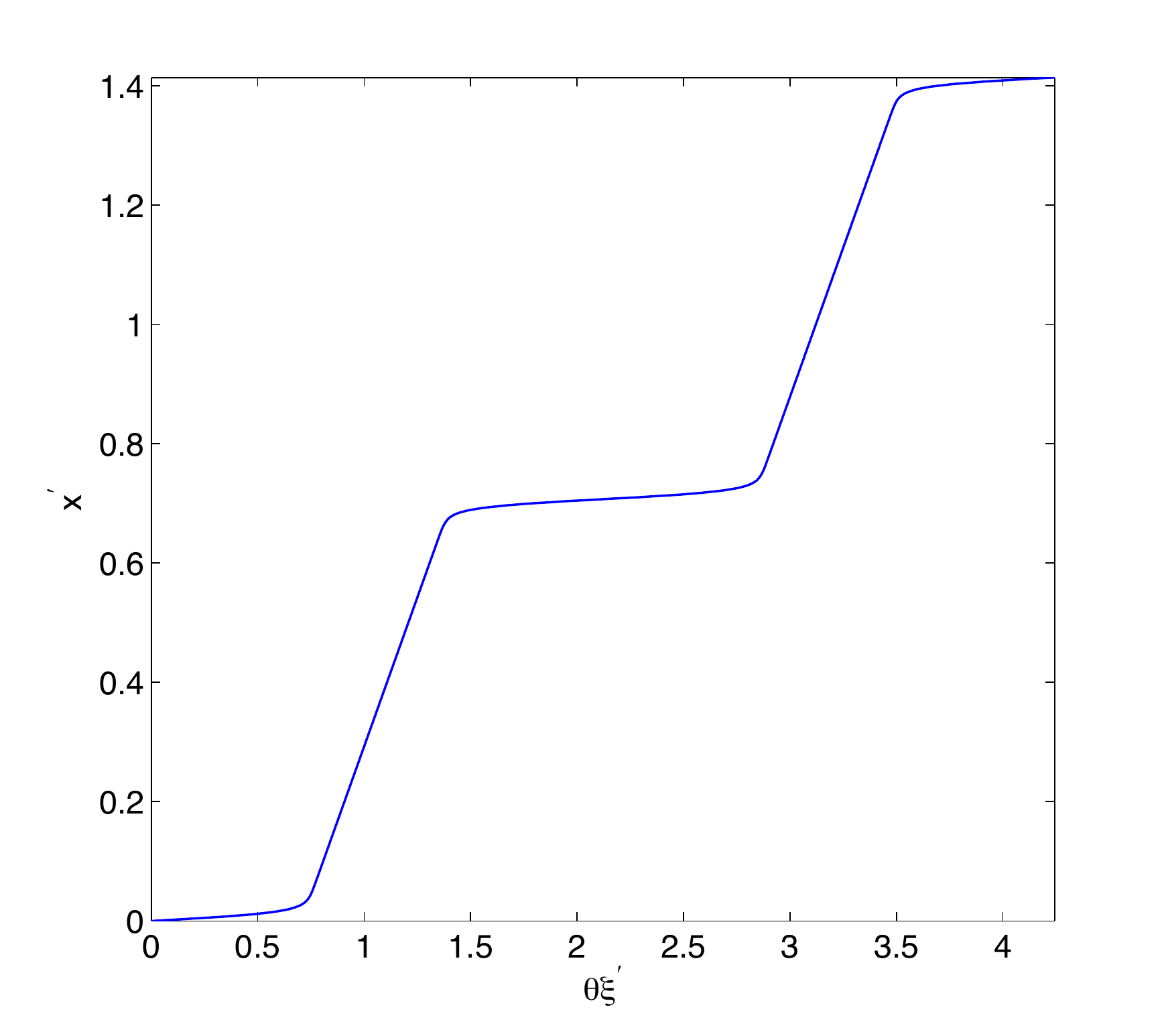}\end{center}
\caption{The function $R_1^{-1}$ for Case 1, $\theta=3 + {\cal O}(\exp(-50))$.}
\label{R1}
\end{figure}
Observe that this function is very flat close to $x'=0,1/\sqrt{2},\sqrt{2}$, and mesh points will be
concentrated at these values.

It follows immediately that  $R_2(y') = y',$ and also
$$\xi' = (\xi + \eta)/\sqrt{2}, \quad \eta'=(\xi-\eta)/\sqrt{2}, \quad x = (x'+y')/\sqrt{2}, \hspace{.1cm}\mbox{and} \hspace{.1cm} y=(x'-y')/\sqrt{2}.$$
Therefore, from (\ref{xstuff}) and (\ref{ystuff}) it follows that
$$
x=\frac{1}{\sqrt{2}}[R_1^{-1}(\theta(\xi +\eta)/\sqrt{2})-((-\xi +\eta)/\sqrt{2})],
$$
$$
y=\frac{1}{\sqrt{2}}[R_1^{-1}(\theta(\xi +\eta)/\sqrt{2})+((-\xi +\eta)/\sqrt{2})],
$$
where $\theta$ is given by (\ref{thec}). A plot of the resulting mesh is shown in Fig. \ref{shocke1}(a) with a close-up in Fig. \ref{shocke1}(b). This mesh is the image of a uniform square computational mesh and has the 
points $(x(\xi_i,\eta_j),y(\xi_i,\eta_j))$, where $\xi_j=\eta_j=j/(n-1)$, for $i,j=0,...,N-1$ and $N=60$.
\begin{figure}[hhhhhhhhhhhhhhhhhhhhhhhhhhhhhhhhhhhhhhhhhhhhhhhhhhhhhhhhhhhhhhhhhhhhhhhhhhhhhhhhhhhhhhhhhhhh]
\begin{center}\includegraphics[height=5.5cm,width=6cm]{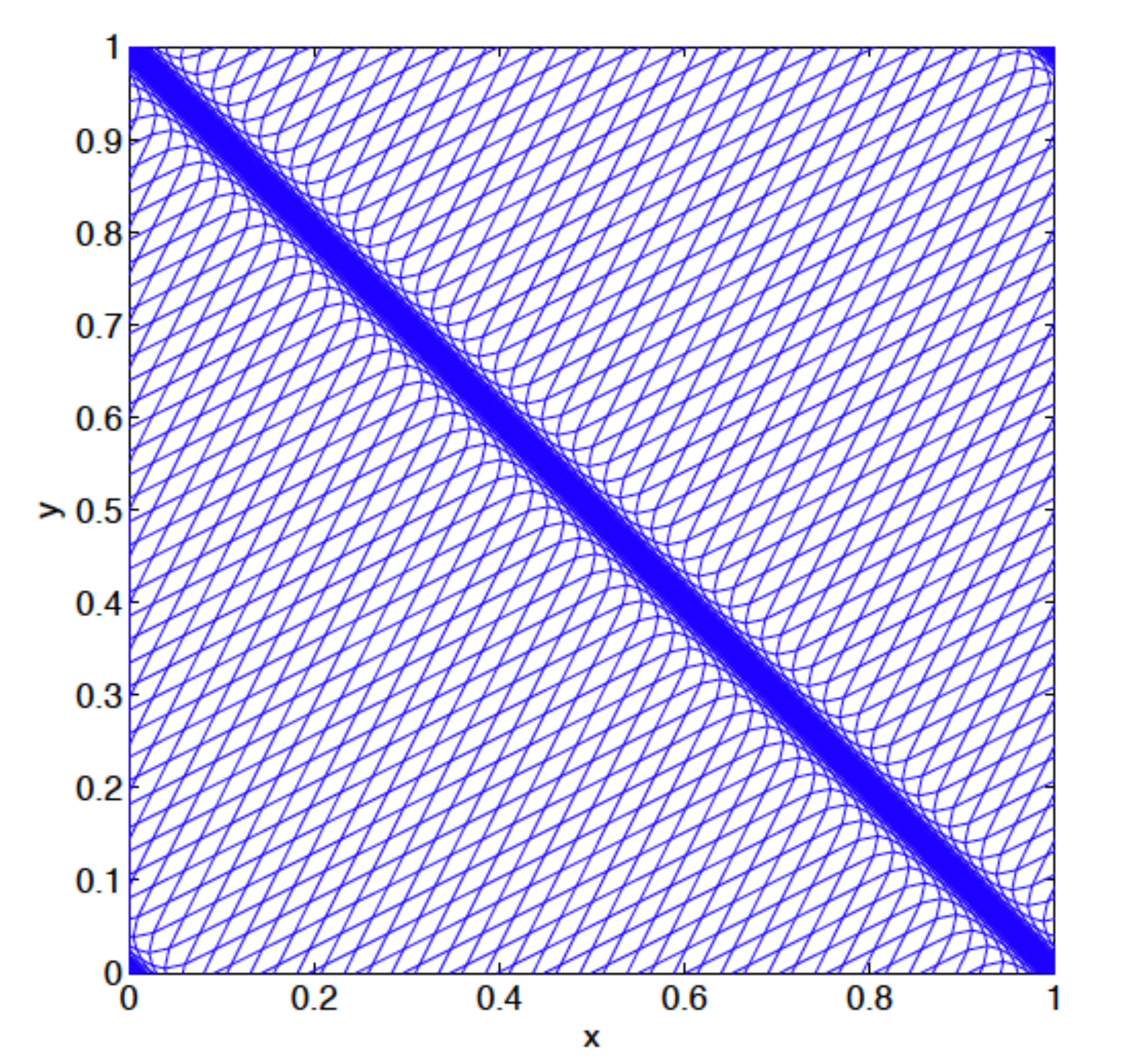}\includegraphics[height=5.5cm,width=6cm]{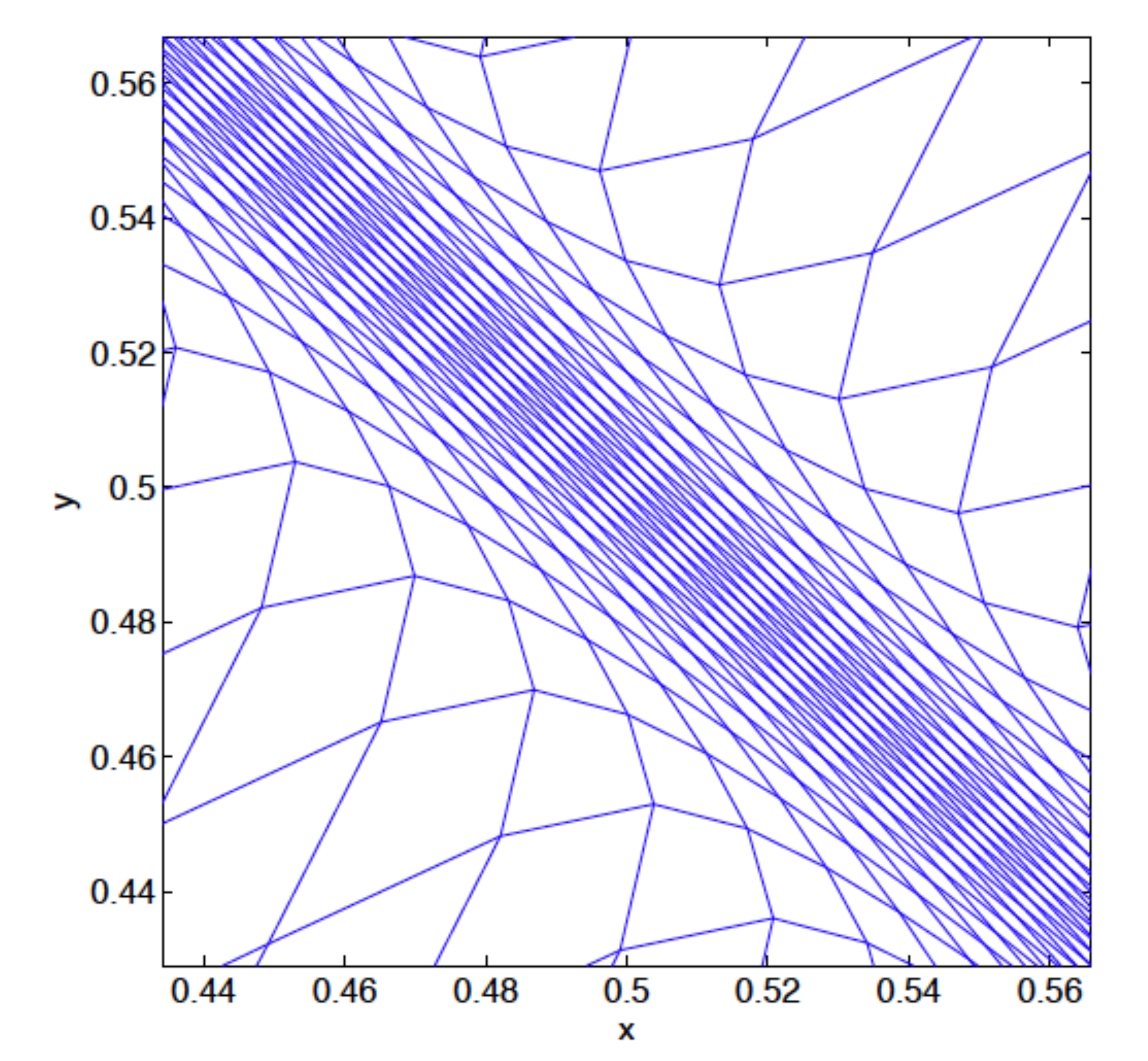}
\caption{(Left) A $(60\times60)$ mesh generated from the analytical solution of the Monge Ampere equation for the density function in Case 1. (Right) A zoom of the region along the shock where the density function is large.}\label{shocke1}\end{center}
\end{figure}
We see that not only is the mesh concentrated along the linear features parallel to ${\mathbf e}_2$ but it is also closely
aligned with this vector. Away from the linear feature the mesh has a distinctive diamond shape, with each diamond of uniform size and with axes  in the directions
${\mathbf e}_1$ and ${\mathbf e}_2$. The close-up shows the diamonds stretched along the linear feature and then smoothly evolving into 
uniform diamonds. 

The alignment properties of the mesh can be calculated directly from the Jacobian.  The eigenvalues of ${\mathbf J}$ (which coincide with the singular values) 
are given from (\ref{spotty1}) by $\lambda_1 = \theta/\rho$ and $\lambda_2 = 1$. Ignoring exponentially small terms, we have $\lambda_1 =  3/51$ within the linear feature, and $\lambda_1 = 3$
away from the linear feature, implying that the alignment measure $Q_s$ in (\ref{qske}) is 
\begin{equation}
Q_s = 8.529 \quad \mbox{within linear feature}, \quad Q_s = 1.667 \quad \mbox{outside linear feature}.
\label{qnums}
\end{equation}
In contrast, by construction the M-alignment measure $Q_a = 1$ at all mesh points.
\subsubsection{Example 2: Two orthogonal shocks}
Consider orthogonal shocks of different widths and magnitudes with the associated scalar density
\[
\rho(\mathbf{x})=\rho_1(x^{\prime})\rho_2(y^{\prime}).
\]
Here $\rho_1(x')$, $\theta_1$, and $R_1(x')$ 
are the same as in Example 1, and 
\[\rho_2=1+10\sum_{m=-\infty}^{\infty}\mathrm{sech}^2(25(\sqrt{2}y^{\prime}-m)).
\]
A direct calculation gives $\theta_2=1.8 + {\cal O}(\exp(-25))$, and
\[
R_2(y^{\prime})=y^{\prime}+\frac{\sqrt{2}}{5}\sum_{m=-\infty}^{\infty}[\tanh(25(\sqrt{2}y^{\prime}-m))-\tanh(-25m)].
\]
The inverse of $R_2$ can be computed in the same manner as for $R_1$ in the previous case and has the same qualitative structure.

This density $\rho(\mathbf{x})$ concentrates mesh points along linear features parallel to ${\mathbf e}_1$ and ${\mathbf e}_2$, and
using the same procedures as in Example 1, the mesh is computed as
\begin{eqnarray}
x&=&\frac{1}{\sqrt{2}}[R_1^{-1}(\theta_1(\xi +\eta)/\sqrt{2})-R_2^{-1}(\theta_2(-\xi +\eta)/\sqrt{2})],\nonumber\\
y&=&\frac{1}{\sqrt{2}}[R_1^{-1}(\theta_1(\xi +\eta)/\sqrt{2})+R_2^{-1}(\theta_2(-\xi +\eta)/\sqrt{2})].\nonumber
\end{eqnarray}
 A plot of the image of a uniform mesh under this map is shown in Fig.\ref{shocke3},
\begin{figure}[hhhhhhhhhhhhhhhhhhhhhhhhhhhhhhhhhhhhhhhhhhhhhh]
\includegraphics[height=5.5cm,width=6cm]{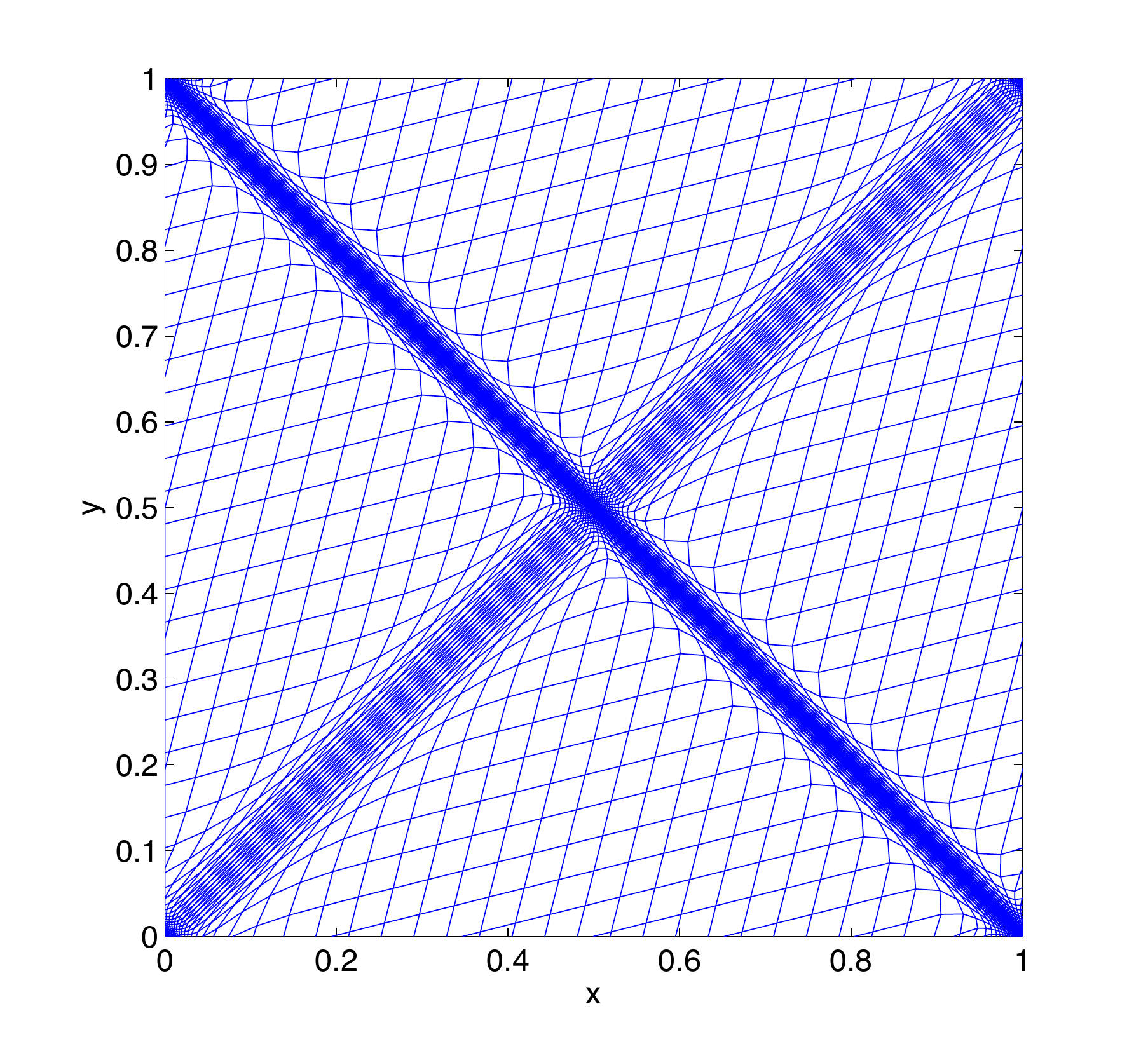}\includegraphics[height=5.5cm,width=7.2cm]{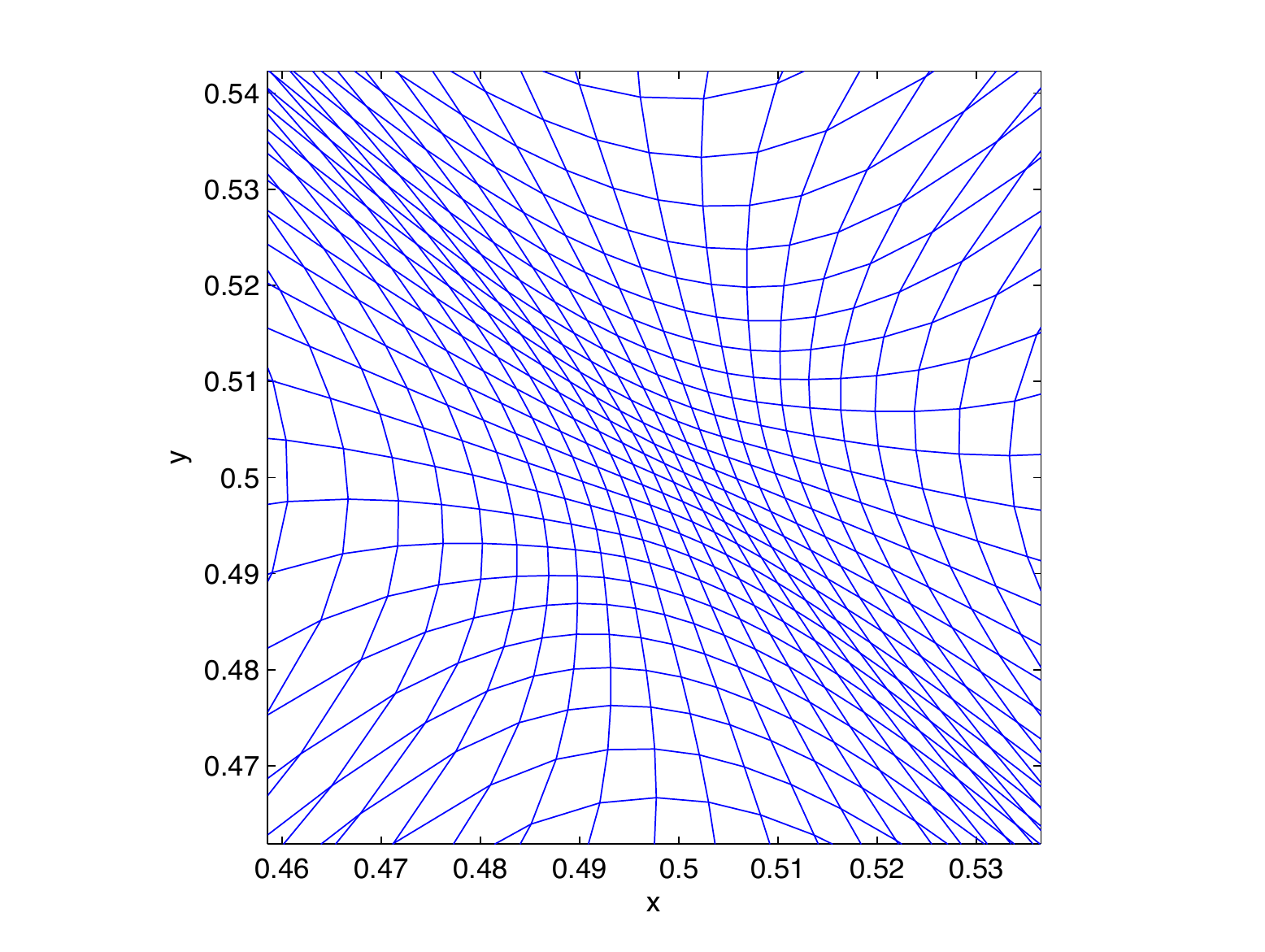}
\caption{(Left) A $(60\times60)$ mesh generated from the analytical solution of the Monge Ampere equation for the density function in Case 2. (Right) A zoom of the region along the shock where the density function is large.}
\label{shocke3}
\end{figure}
where we see the excellent alignment of the mesh to the two linear features. Note also the very smooth transition of the mesh
from one feature to the other. The eigenvalues  $\lambda_1$ and $\lambda_2$ (up to exponentially small terms) are given by:
\begin{enumerate}
\item First linear feature alone: \quad $\lambda_1=3/51, \quad \lambda_2 = 1.8, $
\item Second linear feature alone: \quad $\lambda_1 = 3, \quad \lambda_2 = 1.8/11$
\item Intersection of the two linear features: $\lambda_1 = 3/51, \quad \lambda_2 = 1.8/11$
\item Outside the two linear features: $\lambda_1 = 3, \quad \lambda_2 = 1.8.$
\end{enumerate}
The respective values of the alignment function $Q_s$ are 
\begin{equation}
1. \quad Q_s  = 15.31, \quad 2. \quad Q_s =  9.19, \quad 3. \quad Q_s = 1.57, \quad 4. \quad Q_s = 1.13.\label{qnums2}
\end{equation}
We deduce that away from the linear features and also in the intersection of the two features the mesh in 
Example 2 is less skew than that of Example 1.
\section{Numerical Examples using Parabolic Monge-Amp\`ere algorithm}
The examples described in the previous section relate to problems in which we can exactly solve the Monge-Amp\`ere equation.
For most problems we must instead use a numerical method to compute the solution of  this nonlinear elliptic PDE
together with its associated boundary conditions.  Although various methods have been proposed for solving the Monge-Amp\`ere equation directly for the purpose of mesh generation \cite{Finn}, \cite{froese}, a method that is both cheap and reliable is  relaxation in which the Monge-Amp\`ere equation 
is solved as the limit of the explicit solution of an associated parabolic equation. This is implemented as the 
Parabolic Monge-Amp\`ere algorithm (PMA)  \cite{BW:09}, \cite{BW:06}, \cite{walsh},\cite{Pab}.  In this section we give a series of four calculations using the PMA algorithm The first two are simply
numerical calculations of the two exact solutions described in the previous section. These
computations show clearly the convergence to the unique solution obtained in
Section 4. We will also consider various numerical measures of alignment. The second two examples look at problems in which the features are non-orthogonal or have
significant curvature.
\subsection{Example 1: Parallel linear shocks} Consider parallel linear shocks in a periodic domain, with $\rho$ given as in Example 1 of Section 4.
The  mesh generated by using the PMA algorithm is shown on the left in Fig.\ref{mesh1LS} and closely corresponds to that given from the exact solution of the Monge-Amp\`ere equation presented in Fig \ref{shocke1}. 
The figure on the right depicts the ellipses (in blue) formed from circumscribing the eigenvectors of $\mathbf{J}$, the lengths of which are scaled by their associated eigenvalues.
The alignment properties of the mesh are clear from these ellipses.
\begin{figure}[hhhhhhhhhhhhhhhhhhhhhhhhhhhhhhhhhhhhhhhhhhhhhhhhhhh!!!!!]
\begin{center}\includegraphics[height=5.5cm,width=6.6cm]{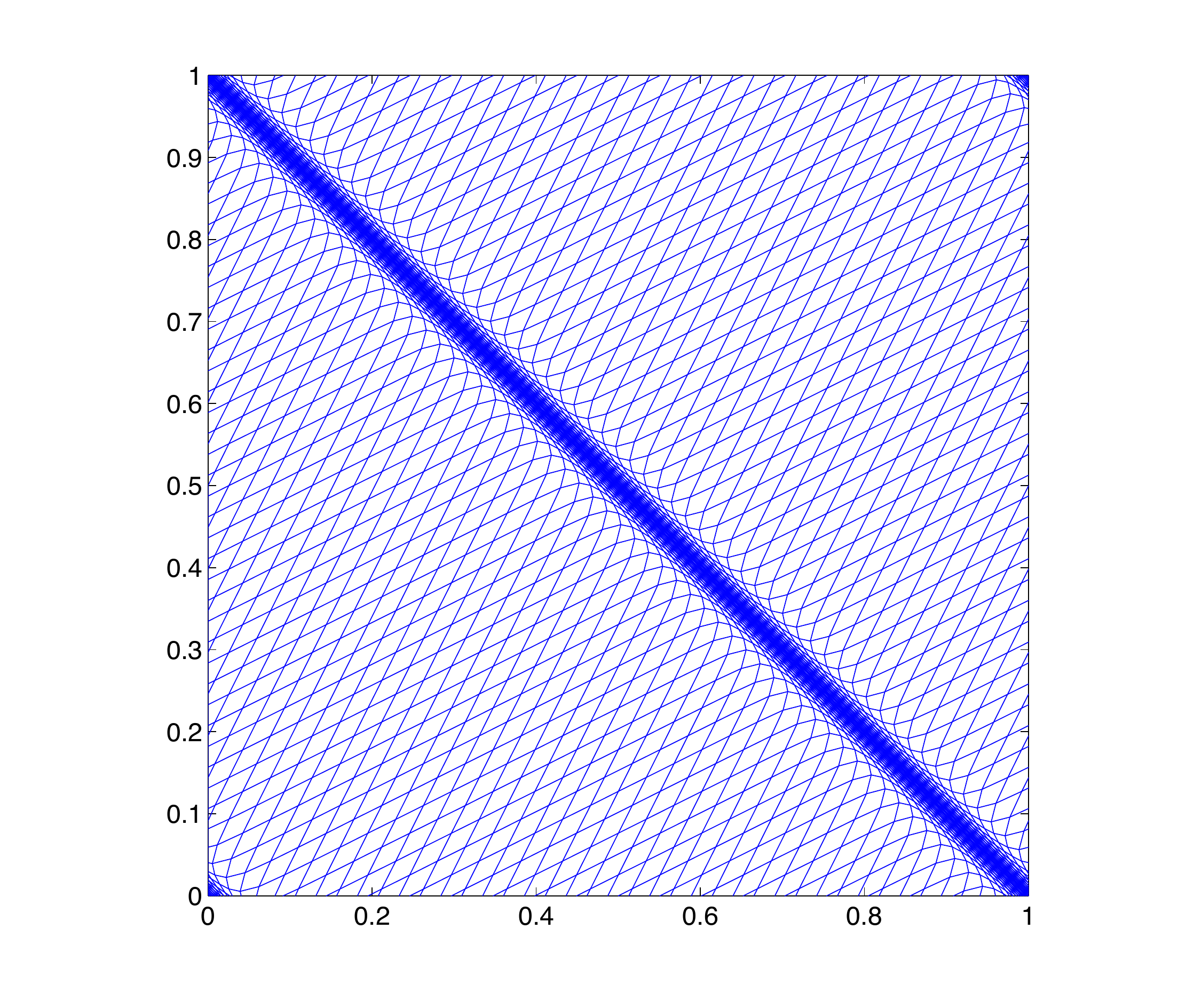} \includegraphics[height=5.5cm,width=6cm]{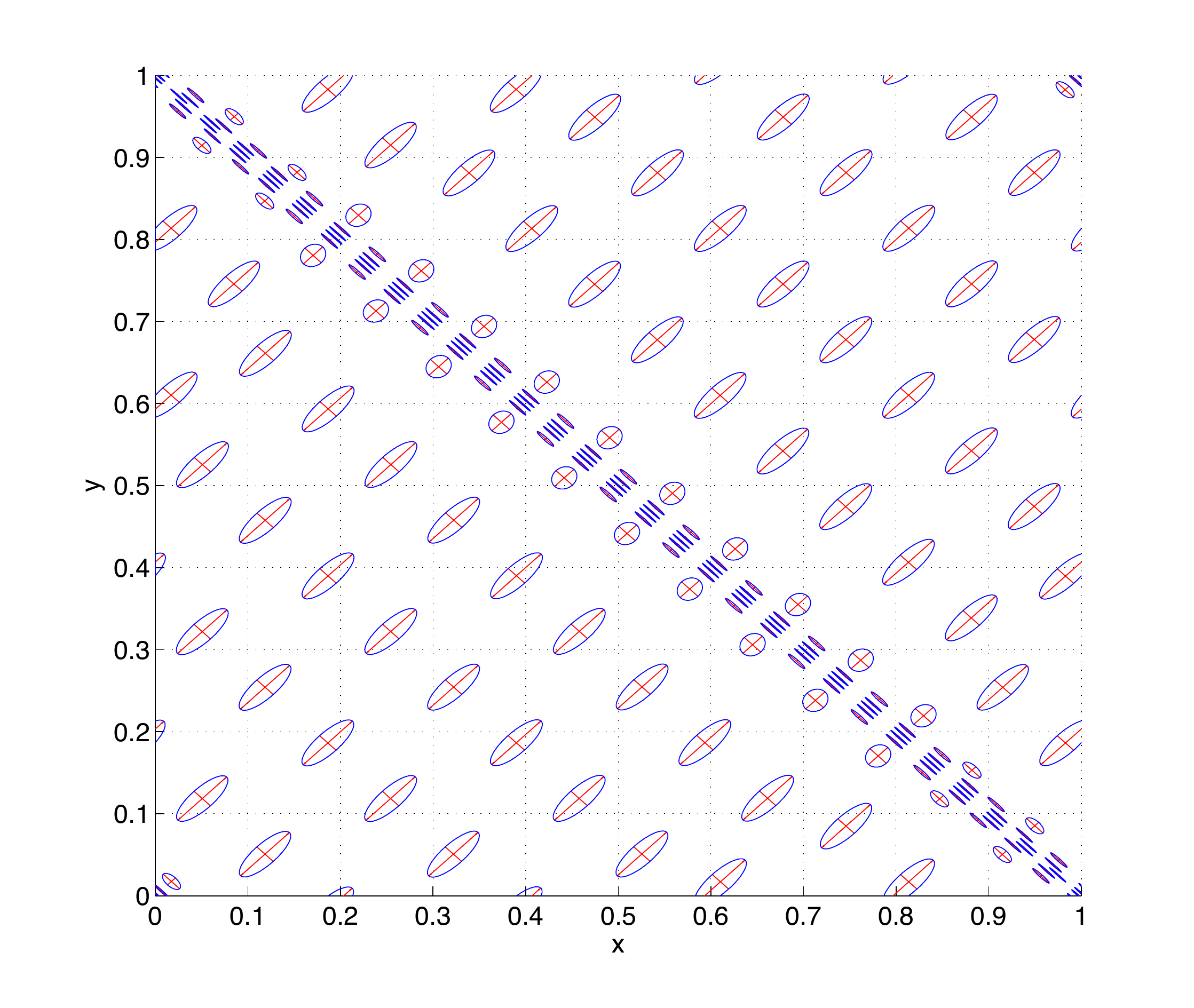}\\
\caption{ \small Example 1: The PMA mesh with (60 $\times$ 60) mesh points (left), and ellipses from eigensystem of associated Jacobian Matrix (right).}\label{mesh1LS}\end{center}
\end{figure}
The value of $Q_s$ computed within the linear feature is  $8.364$ and outside the linear feature is $1.669$, which compares well with the analytical results in (\ref{qnums}). Furthermore, the M-alignment measure 
$Q_a$ ranges from $1$ to $1.017$, showing the close alignment
of the mesh calculated using the PMA algorithm to the Metric Tensor ${\mathbf M}$ in (\ref{metex1}).
 \subsection{Orthogonal shocks of different density}We now consider a pair of orthogonal shocks of different densities as defined in Example 2. 
In Fig.\ref{mesh3new} the PMA mesh is shown on the left and ellipses formed from eigensystems of its associated Jacobian matrix $\mathbf{J}$ on the right. 
\begin{figure}[hhhhhhhh!!!!!]
\includegraphics[height=5.5cm,width=6.2cm]{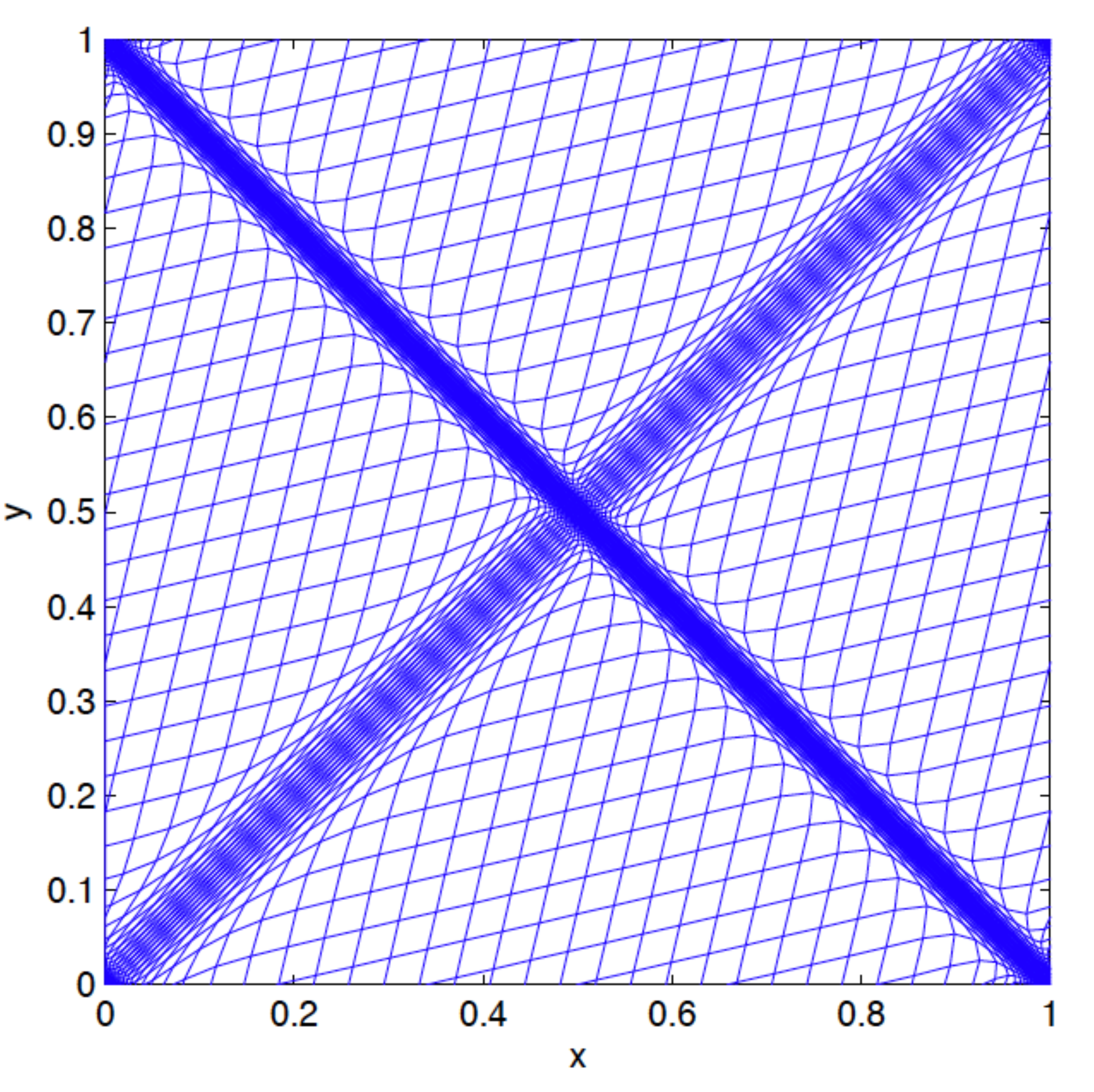} \includegraphics[height=5.7cm,width=6.5cm]{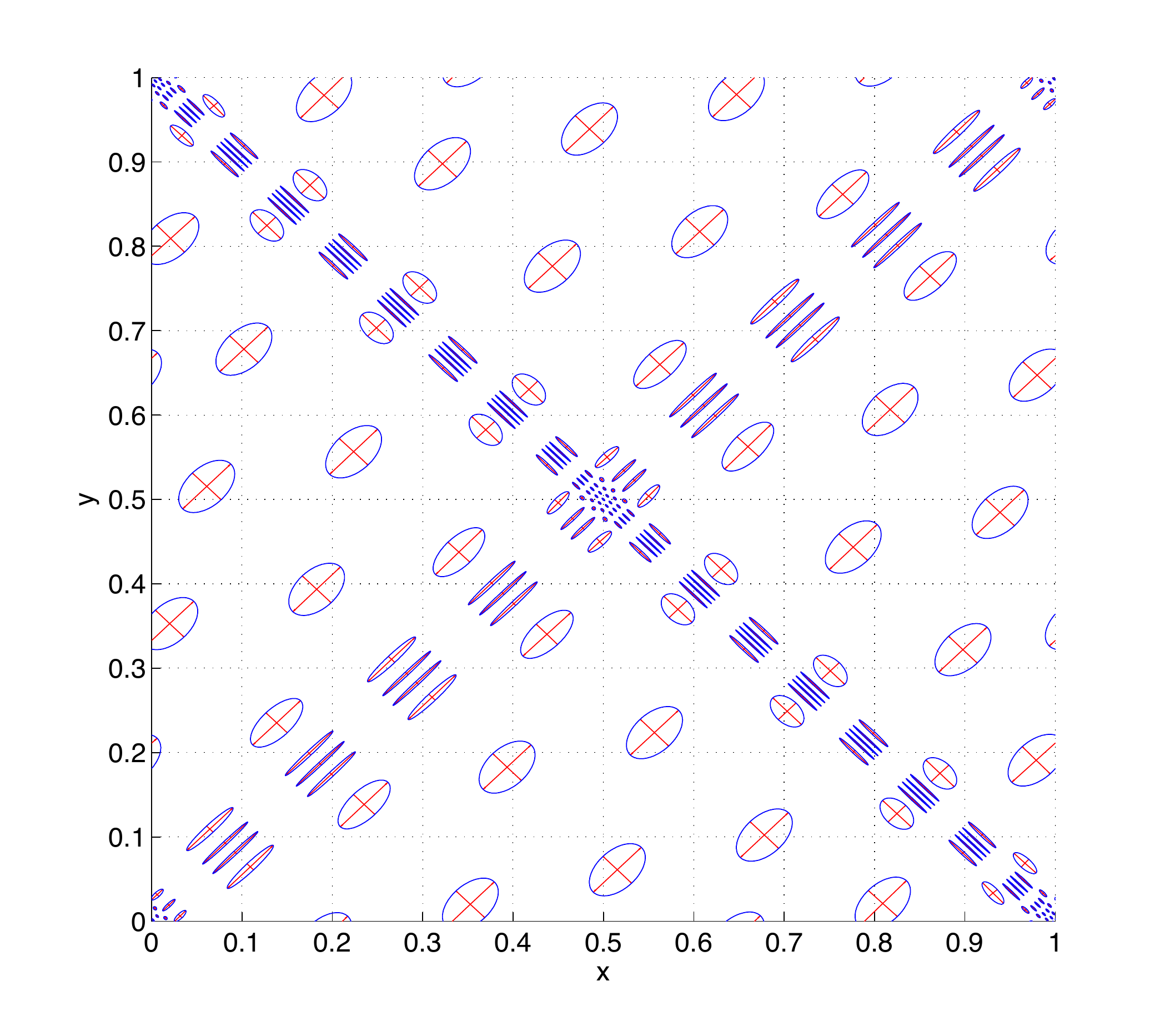}\\
\caption{ \small Example 2: The mesh calculated using the PMA algorithm with 60 $\times$ 60 mesh points (left), and ellipses from eigensystem of associated Jacobian Matrix (right).}
\label{mesh3new}
\end{figure}
The values of $Q_s$ within each linear feature alone, at the intersection of the linear features, and outside the two linear features are respectively
$$1. \quad Q_s  = 15.06, \quad 2. \quad Q_s =  9.04, \quad 3. \quad Q_s = 1.59, \quad 4. \quad Q_s = 1.14,$$
which again compare well with the values of  $Q_s$ given in (\ref{qnums2}). The calculated alignment measure 
$Q_a$ ranges from $1$ to $1.028$, again demonstrating that the PMA mesh aligns well to the Metric Tensor in (\ref{metex3}) and the Jacobian matrix given in (\ref{jacex34}).
\subsection{Non-orthogonal shocks}
In this example we consider a problem with non-orthogonal intersecting shocks defined by the scalar density function $\rho=\rho_1\rho_2$
\begin{eqnarray}
\rho_1&=&1+50\sum\limits^{-1}_{i=1}\mathrm{sech}^2(50( \sqrt{2}x^{\prime}-i))),\quad \rho_2=1+10\sum\limits^{3}_{i=1}\mathrm{sech}^2(25( \sqrt{5}y^{\prime}-(2i-1))),\nonumber\\
x^{\prime}&=&-\frac{x}{\sqrt{2}}+\frac{y}{\sqrt{2}},\quad y^{\prime}=\frac{x}{\sqrt{5}}+\frac{2y}{\sqrt{5}}.\nonumber
\end{eqnarray}
The mesh calculated using the PMA algorithm and the associated ellipses are shown in Fig.\ref{mesh4}. 
\begin{figure}[hhhhhhhh!!!!!]
\includegraphics[height=5.5cm,width=6.2cm]{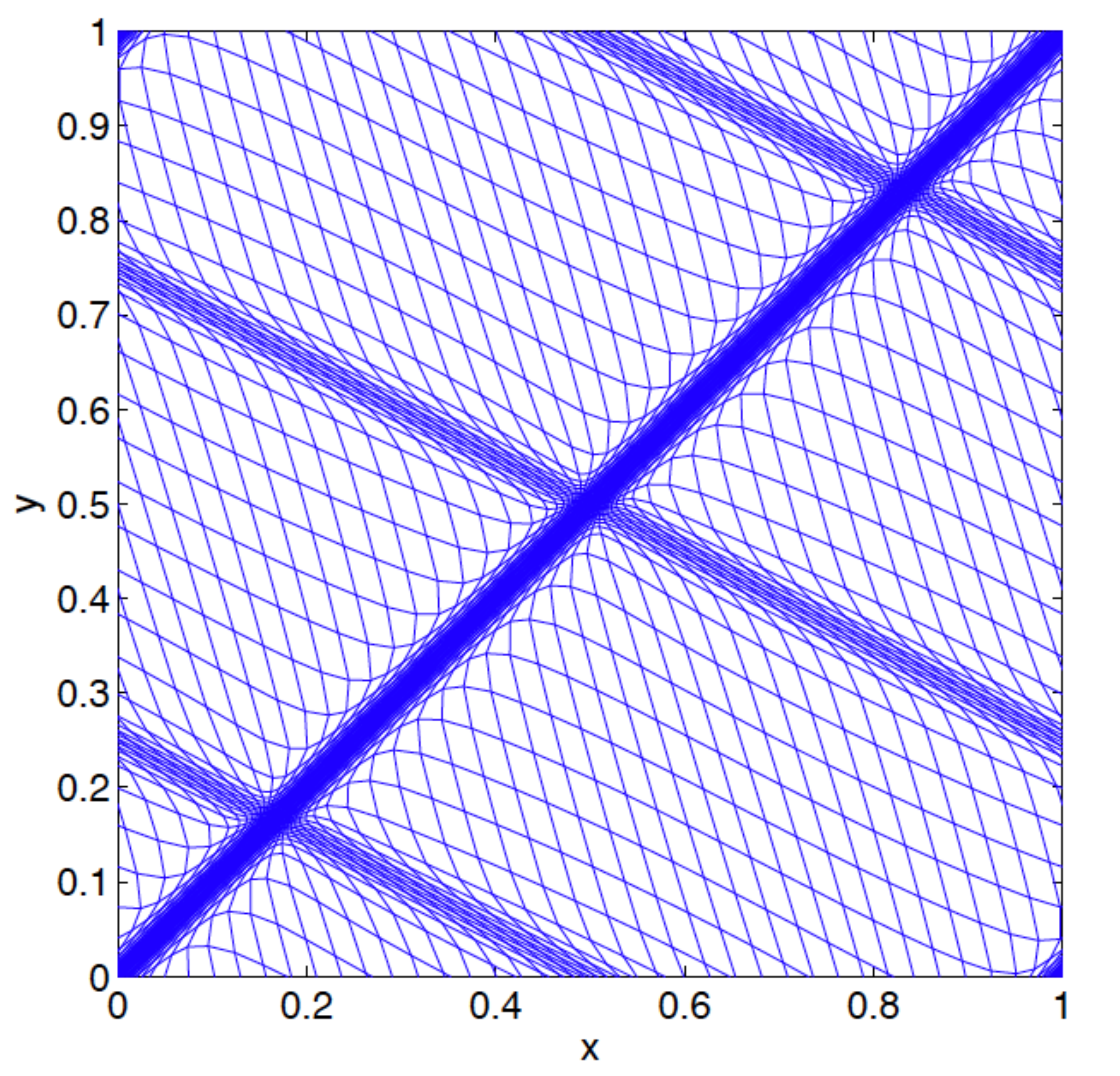} \includegraphics[height=5.8cm,width=6.5cm]{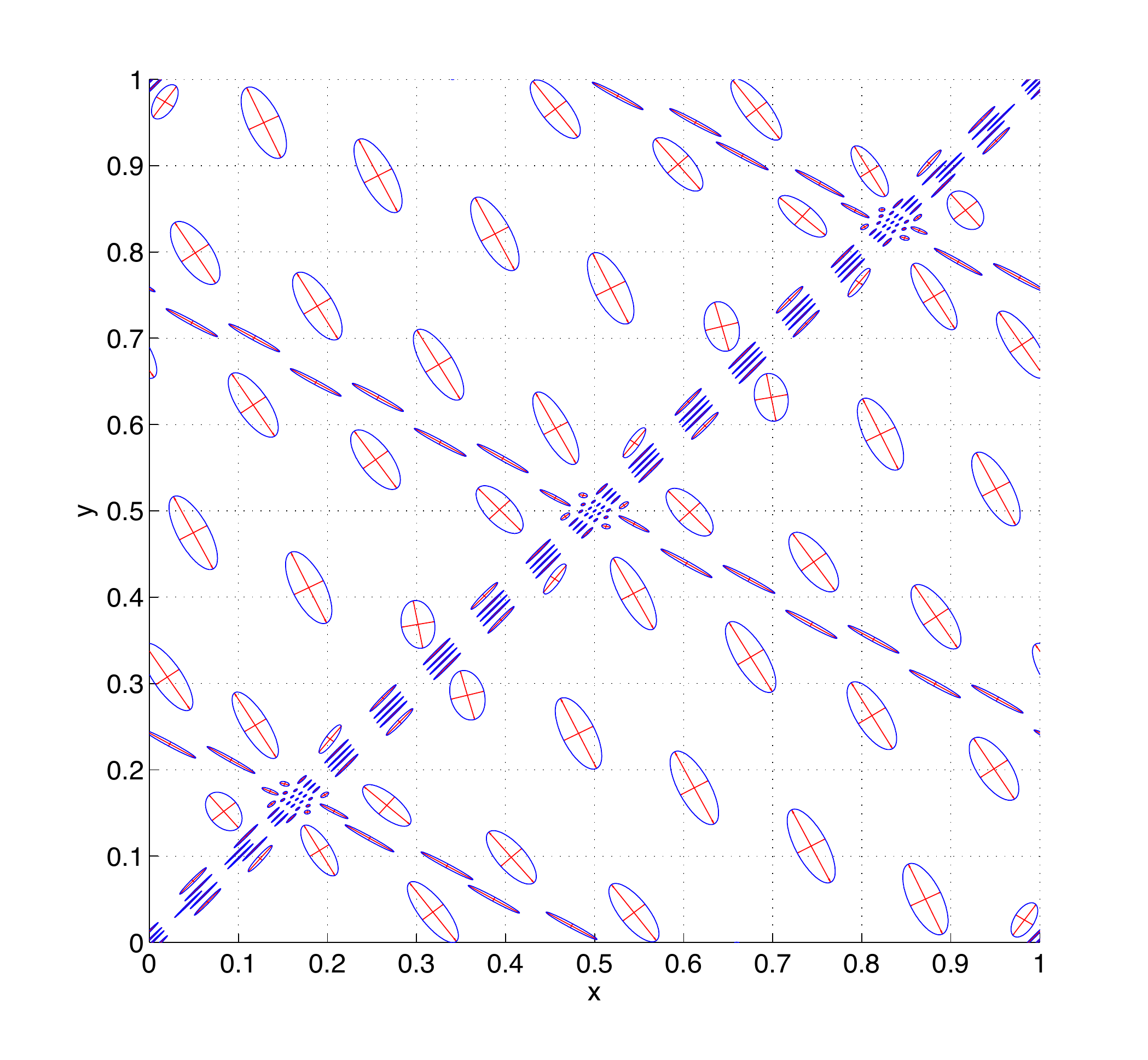}\\
\caption{ \small Example 3: The PMA mesh with 60 $\times$ 60 mesh points (left), and ellipses from eigensystem of associated Jacobian Matrix (right).}
\label{mesh4}
\end{figure}
The case of non-orthogonal shocks is more problematic to analyse as the PMA solution is not separable in this case and we have no exact 
analytic solution. 
However, we expect that locally along each shock, away from the intersection, the mesh aligns to the same metric as that derived for a single shock of similar magnitude. Consequently, if we consider a mesh element where $\rho_1$ is large and $\rho_2\approx1$, then we expect the mesh to align to a Metric Tensor $\mathbf{\tilde{M}} \approx \mathbf{M}$ with eigenvalues and eigenvectors given by the expressions (\ref{jacex35c},\ref{spotty2}) so that
\[
\tilde{\mu}_1={\theta_2 \rho_1^2}/{\theta_1}, \hspace{.5cm} \tilde{\mu}_2={\theta_1}/{\theta_2}, \hspace{.5cm} \mbox{and}\hspace{.5cm} \mathbf{\tilde{e}_1}=[-1/\sqrt{2}, 1/\sqrt{2}]^T. 
\]
The plot on the left in Fig. \ref{metric2_ex4} demonstrates that this is indeed a good approximation to the actual metric, where we see the ellipses given by $\mathbf{J}$ along the shock for which $\rho_1$ is large. Note that there is a contribution from the smaller shock defined by $\rho_2$, but this is fairly negligible.  Similarly, if we consider a mesh element where $\rho_2$ is large and $\rho_1\approx1$, then we expect the mesh to align to a metric $\mathbf {\tilde{M}}$ with eigenvalues and eigenvectors given by
\[
\tilde{\mu}_1=\frac{\theta_1 \rho_2^2}{\theta_2}, \hspace{.5cm} \tilde{\mu_2}=\frac{\theta_2}{\theta_1}, \hspace{.5cm} \mbox{and}\hspace{.5cm} \mathbf{\tilde{e}_1}=[1/\sqrt{5}, 2/\sqrt{5}]^T. 
\]
For the symmetric matrix  $\mathbf{\tilde{J}}$ corresponding to the metric $\mathbf {\tilde{M}}$, a comparison of  the ellipses for $\mathbf{\tilde{J}}$ and $\mathbf{J}$
can be seen on the right in Fig. \ref{metric2_ex4},  demonstrating this is also a fairly good approximation to the actual metric in this region. However, the approximation is 
impacted by and not quite as accurate as along the larger shock, with the eigenvectors less orthogonal and tangential to the feature. 
\begin{figure}[hhhhhhhhhhhhhhhhhhhhhhhhhhhhhhhhhhhhhhhhhhhhhhhhhhhhhhhhhhhhhhhhhhhhhhhhhhhhhhhhhhhhhhhhhhhhhhhhhhhhhh]
\includegraphics[height=5.4cm,width=6.2cm]{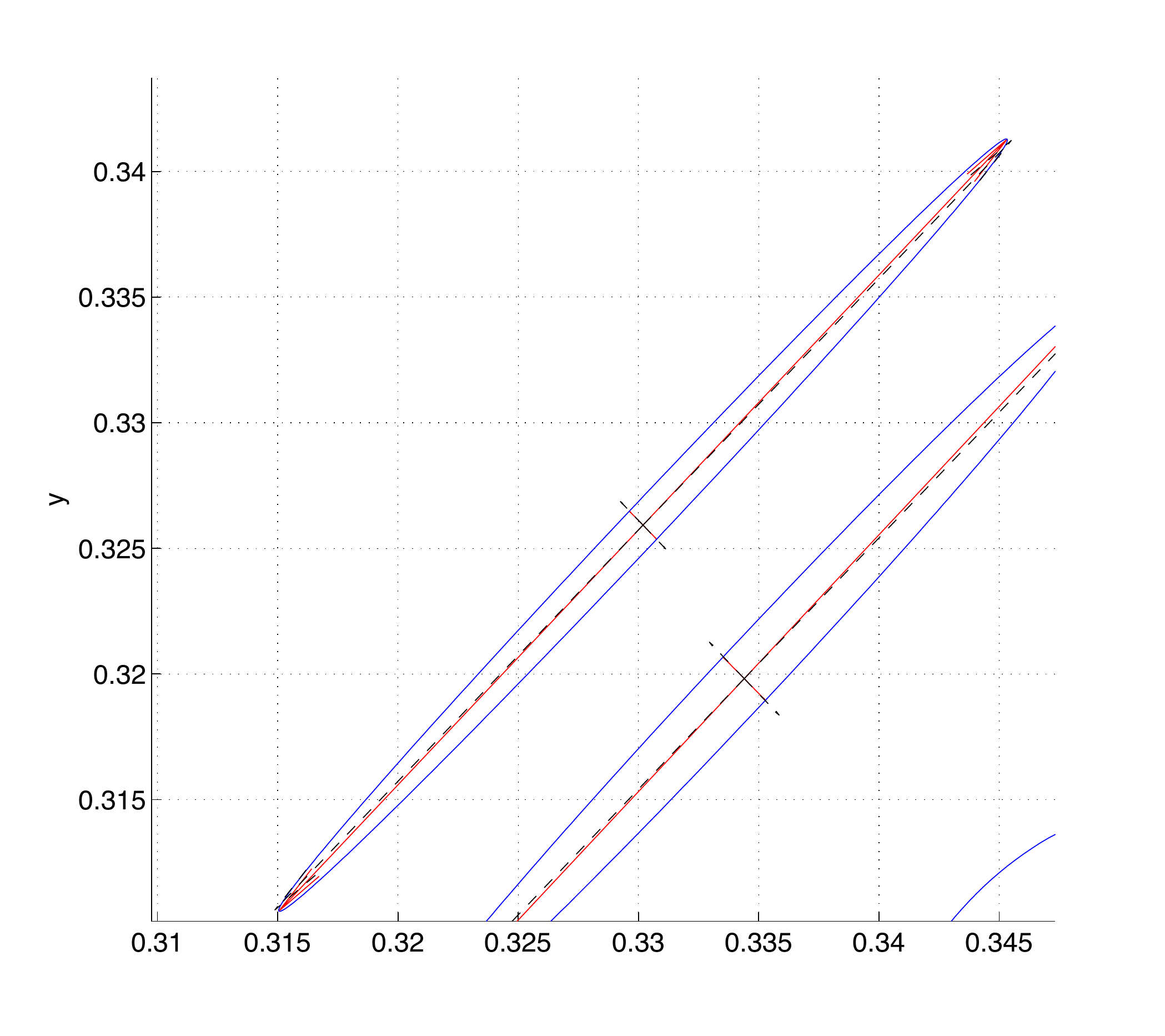}\includegraphics[height=5.4cm,width=6.2cm]{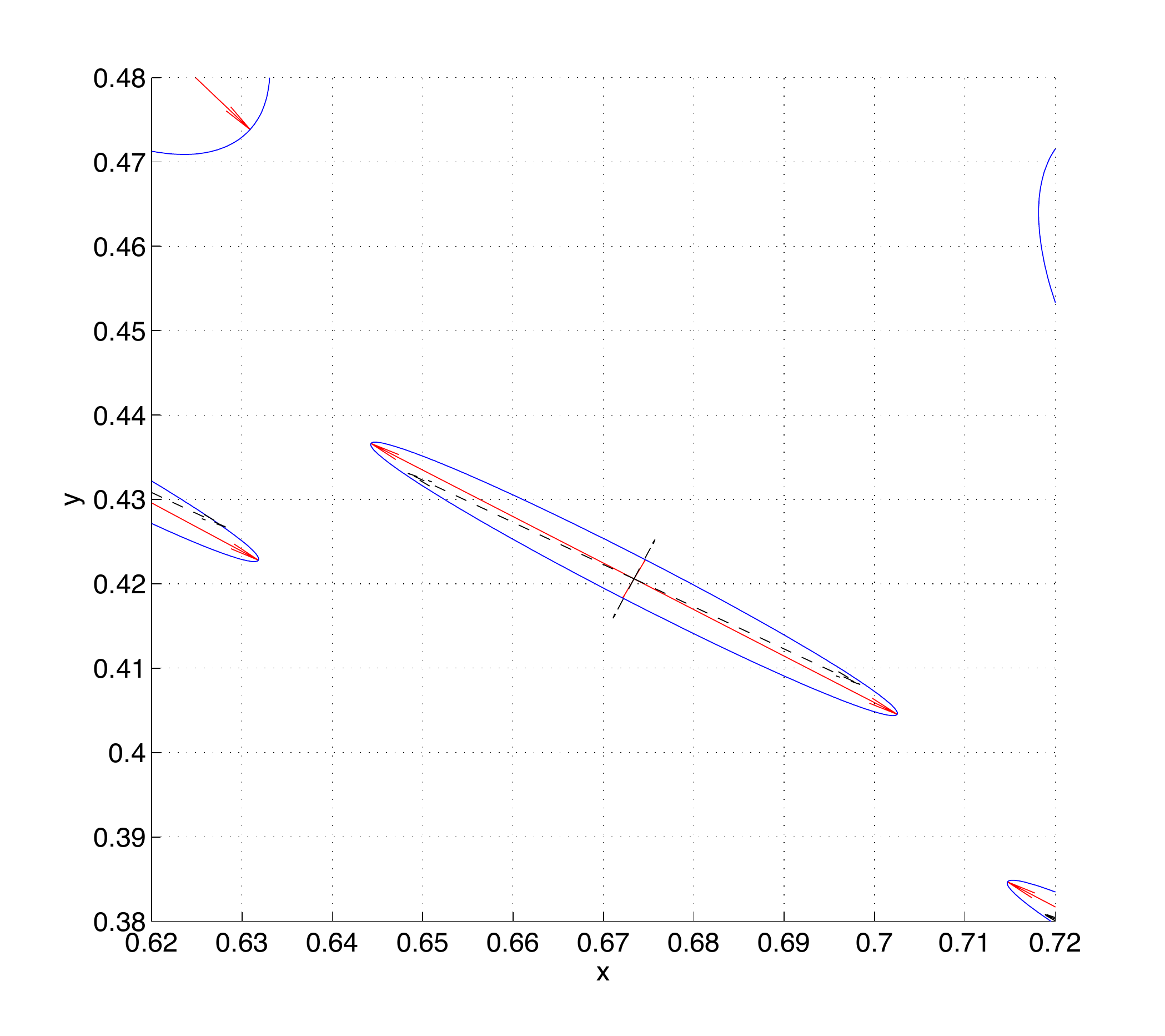}\\
\caption{Example 3: The eigensystems for $\mathbf{J}$ (red) and $\mathbf{\tilde{J}}$ (black), for a mesh element along the shocks for which 
respectively $\rho_1$ (left) and $\rho_2$ (right) is large.}
\label{metric2_ex4}
\end{figure}
\subsection{A nonlinear shock} In this next example we consider a shock concentrated along a nonlinear feature (in this case a sine wave) defined by the scalar density function
\begin{eqnarray}
\rho&=&1+50 \; \mathrm{sech}(50\vert \Psi\vert)^2\nonumber, \hspace{1cm}\Psi=y-0.2 \; \sin(2\pi x)-0.5.\nonumber
\end{eqnarray}
In Fig.\ref{mesh6} the mesh calculated using the PMA algorithm is shown on the left, and the corresponding ellipses for $\mathbf{J}$ and $\mathbf{\tilde{J}}$ are shown on the right.  The symmetric matrix $\mathbf{\tilde{J}}$ corresponds to 
a Metric Tensor $\mathbf {\tilde{M}}$ with eigenvalues and eigenvectors given by
$$ \tilde{\mu}_1=\frac{\rho^2}{\theta}, \hspace{.5cm} \tilde{\mu}_2=\theta, \hspace{.5cm} \mbox{and}\hspace{.5cm} \mathbf{\tilde{e}}_1={\nabla \Psi}/{\| \nabla \Psi \|}.$$  
Here we assume that the orthogonal eigenvectors of $\mathbf {\tilde{J}}$ are in turn orthogonal and tangential to the curve defined as the set for which $\Psi(\mathbf {x})=0$. Given that 
$\rho$ is constant along this curve, it is reasonable to assume there will be no movement of the mesh in that direction, so the eigenvalue corresponding to the tangential eigenvector is chosen to be 1, implying the eigenvalue in the orthogonal direction is $\theta/\rho$. Notice that these eigenvalues correspond to those derived for a single  linear feature where $\Psi=\mathbf{x}\cdot\mathbf{e_1}-c$.  This is a very good approximation 
in the regions along the shock that are close to linear where we observe good alignment to the feature (see also the plot on the left of Fig. \ref{metrics_example6b}). Furthermore, the mesh is close to being uniform away from the feature.
\begin{figure}[hhhhhhhh!!!!!]
\includegraphics[height=5.2cm,width=6cm]{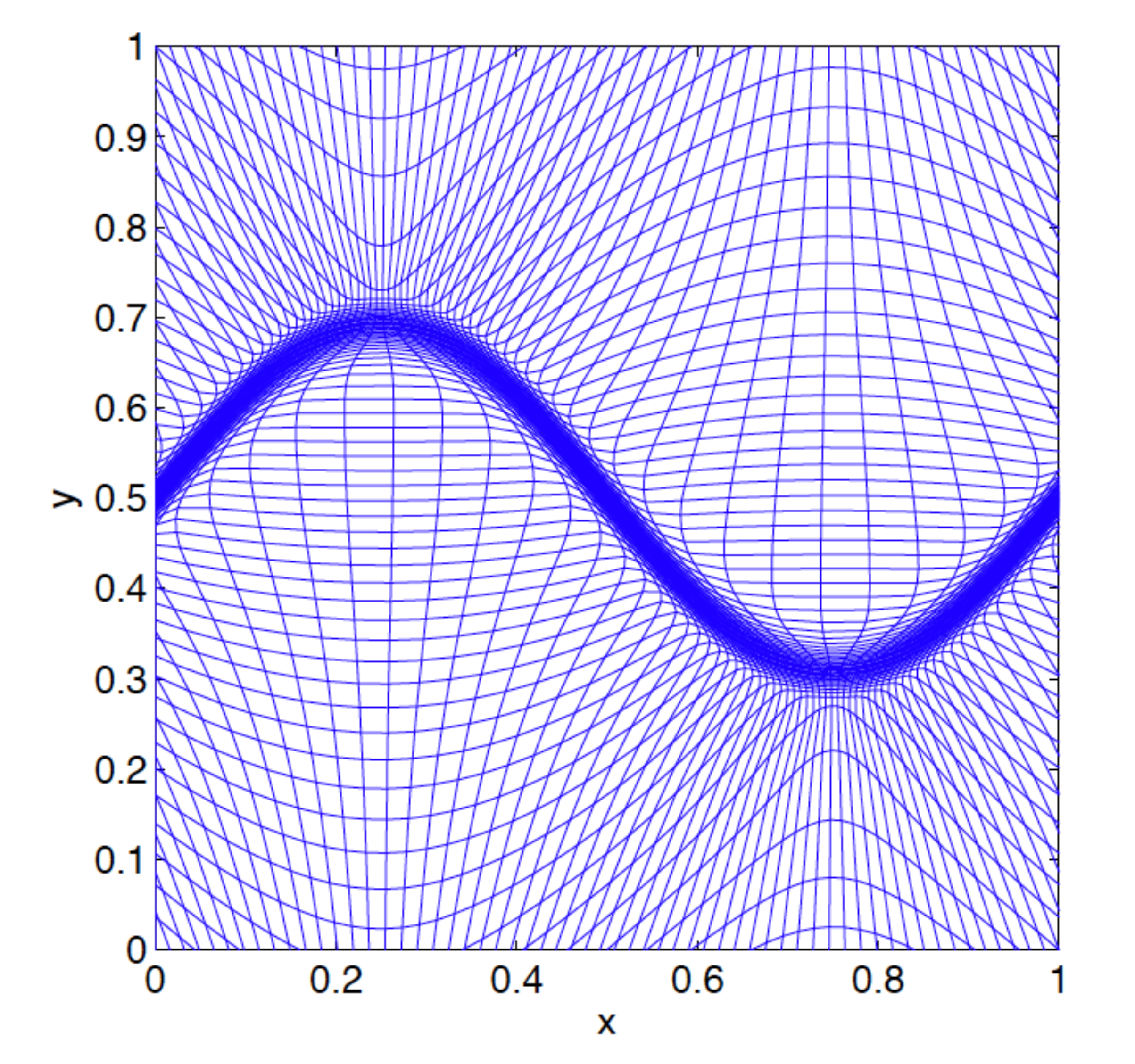}
\includegraphics[height=5.2cm,width=6cm]{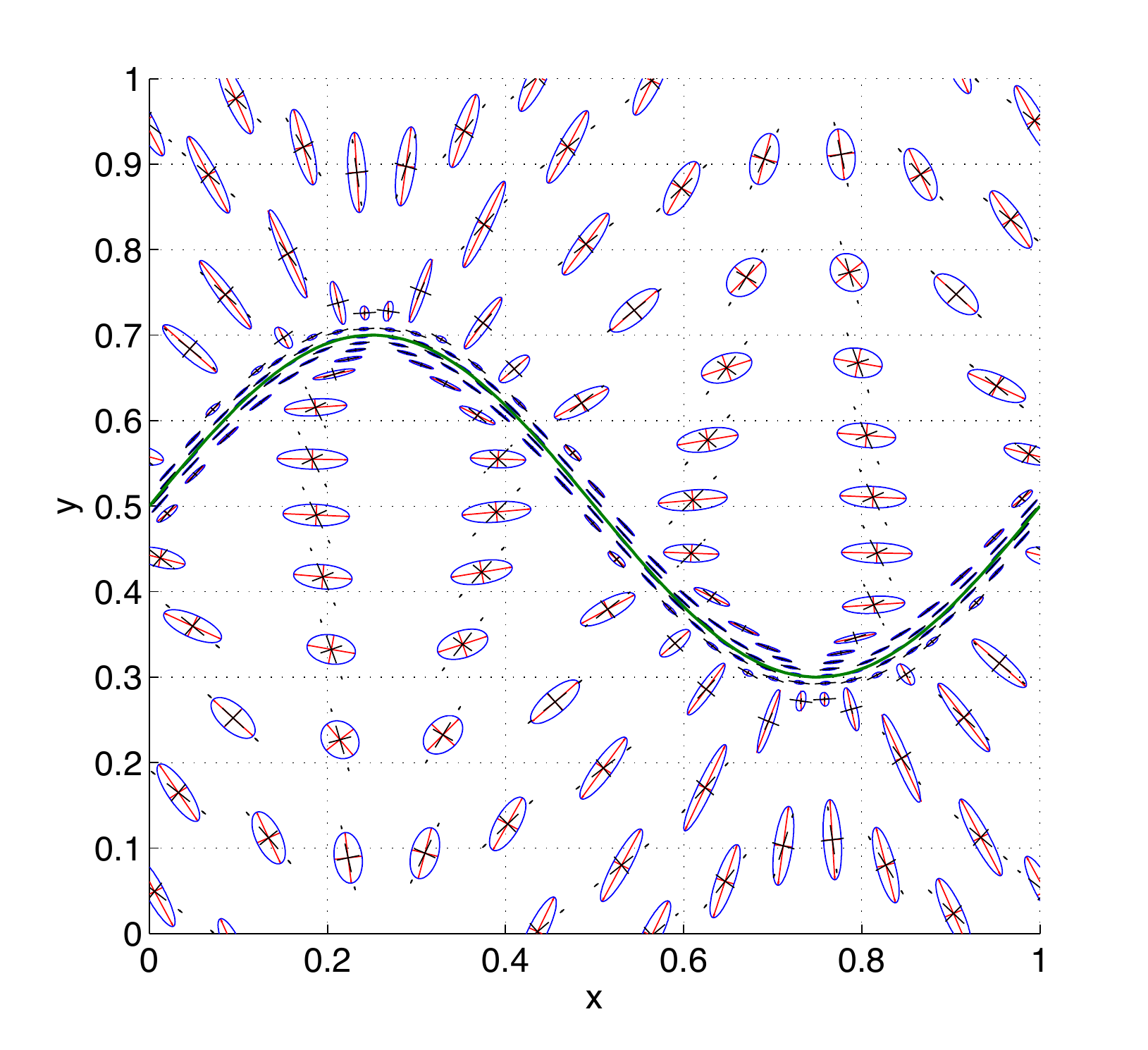}\\
\caption{ \small Example 4: The PMA mesh with 60 $\times$ 60 mesh points (left), and the eigensystems for $\mathbf{J}$ (red) and $\mathbf{\tilde{J}}$ (black) (right). The solid green line represents where 
where $\Psi=0$ and the density function is at a maximum.}
\label{mesh6}
\end{figure}
\begin{figure}[hhhhhhhhHHHHHHHHHHHHHHHHHHHHHHHHHHHHHHHHHHHHHHHHHHHHHHHHHHHHHHHHHHHHHHHHHH!!!!!]
\includegraphics[height=5.2cm,width=6cm]{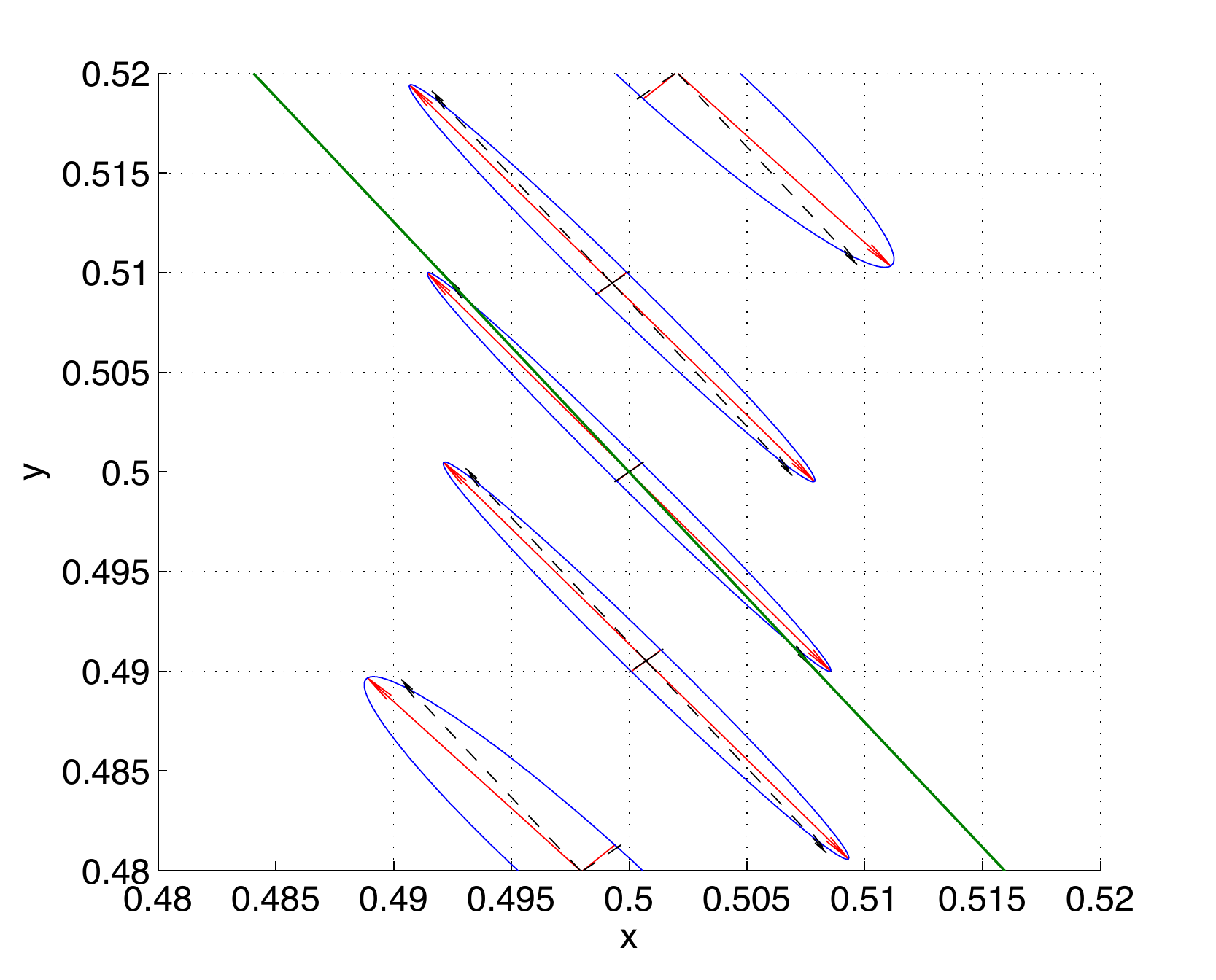}\includegraphics[height=5.2cm,width=6cm]{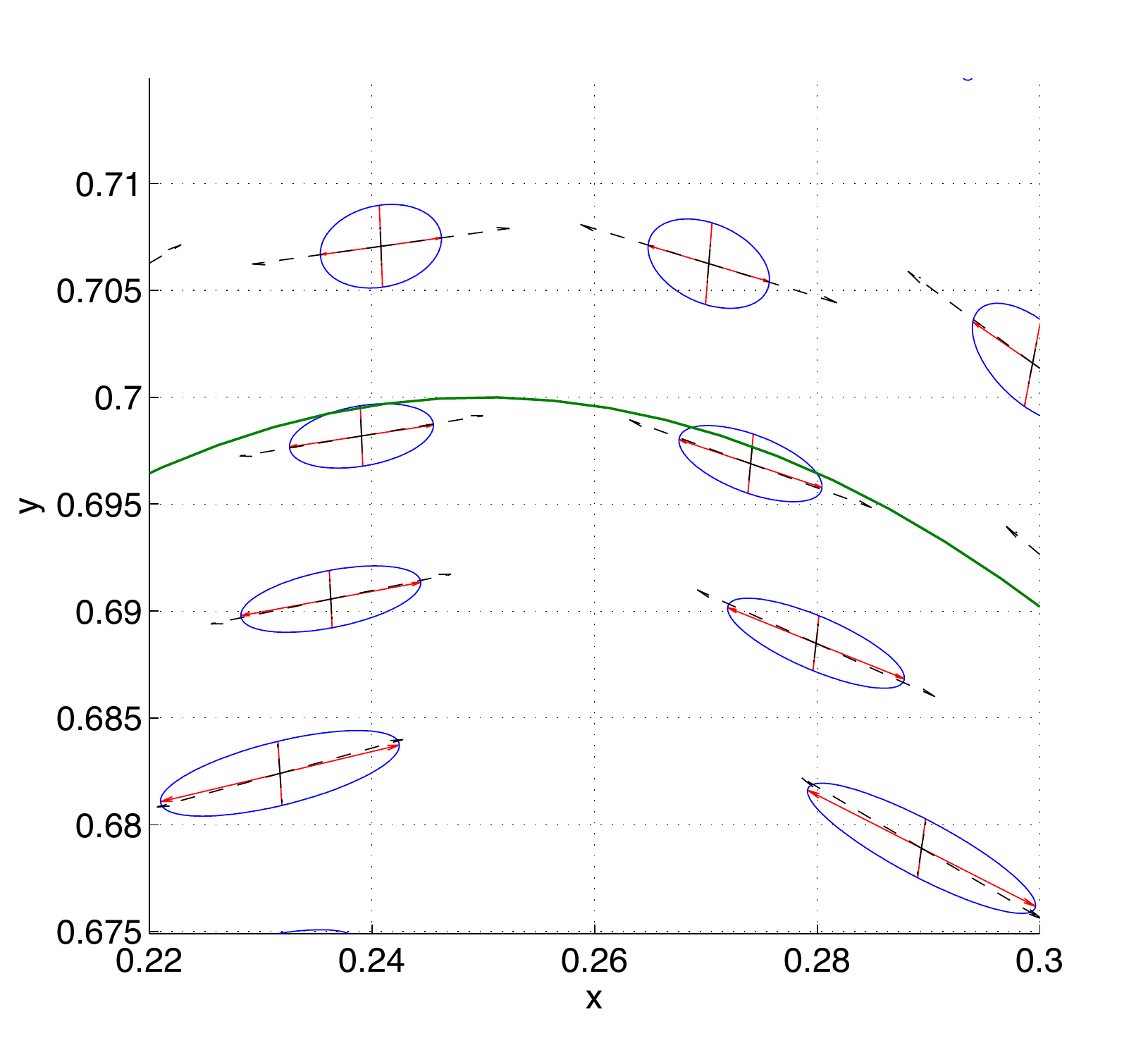}\\
\caption{Example 4: The eigenplot for $\mathbf{J}$ (red) and $\mathbf{\tilde{J}}$ (black) for two different regions along the shock where $\Psi$ has different curvature. The plot on the right depicts a region with more curvature than the plot on the left. In both cases the green solid line depicts where $\Psi=0$.}
\label{metrics_example6b}
\end{figure}
However, in regions with more curvature the mesh elements are less anisotropic as the plot on the right of Fig. \ref{metrics_example6b} demonstrates. Interestingly, the eigenvalues are not well approximated there but the eigenvectors are. In fact we observe that the eigenvectors are approximately tangential and orthogonal to the shock along the entire curve $\Psi({\mathbf x}) = 0$.
\section{Conclusions}
We have shown that a  mesh redistribution method that is based on equidistributing a scalar density function via solving the Monge-Amp\`ere equation has the
capability of producing anisotropic meshes. Furthermore, we have rigorously shown this for a model 
problem comprising orthogonal linear features by
deriving the exact Metric Tensor to which these meshes align. It is quite fascinating that this Metric Tensor has a very similar form to those traditionally used in variational methods. Given that determination of such a tensor
is a difficult task, it would definitely be advantageous if an optimal Metric Tensor arose naturally from the solution of the Monge-Amp\`ere equation. However, a closer examination of how this Metric Tensor is related to those known to minimise interpolation error 
is required. These results have been verified numerically using the Parabolic Monge-Amp\`ere algorithm, a very robust and cheap algorithm. To analyse the level of anisotropy 
we have considered various mesh quality measures and  as in \cite{HRBook} visualised the circumscribed ellipses for eigensystems of Jacobians of mesh mappings
at the mesh elements, which has often
proved much more informative than visualising the meshes themselves. We have also demonstrated that the  
results for the linear case can be used to approximate alignment for more complicated flow structures.  A more rigorous study of features with curvature will form the basis of a subsequent paper.
\section*{Acknowledgments}The authors thank Phil Browne (University of Reading), Mike Cullen (UK Met Office), Weizhang Huang (University of Kansas), and  J F Williams (Simon Fraser University), for many useful suggestions and encouragement. The research of the second and third author is partially supported by NSERC Grant A8781.

\end{document}